%% file: gmres_eval_fin.tex
\documentclass[10pt, a4paper]{article}
\usepackage[a4paper, total={6in, 9in}]{geometry}

\usepackage{amsmath}
\usepackage{amssymb}
\usepackage{amsthm}
\usepackage[algoruled,noline]{algorithm2e}
\usepackage{graphicx}
\usepackage{caption,subcaption}
\usepackage{multirow}
\usepackage{xstring}
\usepackage[utf8]{inputenc}
\usepackage[american]{babel}
\usepackage{amstext,amsbsy,amsopn,eucal,enumerate}
\usepackage{tabularx}
\usepackage[table]{xcolor}
\title{GMRES convergence bounds for eigenvalue problems}
\author{Melina A. Freitag\thanks{Department of Mathematical Sciences,
        University of Bath, Claverton Down, BA2 7AY, United Kingdom, {\tt
m.freitag@maths.bath.ac.uk}}   
\and Patrick K{\"u}rschner\thanks{Max Planck Institute for Dynamics of Complex Technical
Systems,~Sandtorstra{\ss}e~1,~39106~Magdeburg,~Germany,~
\texttt{kuerschner@mpi-magdeburg.mpg.de}.}
\and Jennifer Pestana%
        \thanks{%
                Department of Mathematics and Statistics,
                University of Strathclyde,
                Glasgow, G1 1XQ, United Kingdom
                \texttt{jennifer.pestana@strath.ac.uk}}.
}

\def\R{\mathbb{R}}
\def\C{\mathbb{C}}
\newcommand{\xx}{x_i}
\newcommand{\yy}{y_i}
\newcommand{\ww}{w}
\newcommand{\WW}{W}
\newcommand{\eiga}{\gamma}

\newcommand{\eigb}{\lambda}
\newcommand{\Eigb}{\Lambda}
\newcommand{\minq}{ \min_{\substack{q\in\Pi_k,\\q(0)=1}}}

\newcommand{\cdde}{\texttt{cdde1}}
\newcommand{\cdfd}{\texttt{cd\_fd}}

\newcommand{\olm}{\texttt{olm2000}}
\newcommand{\wt}{w^{(2)}}
\newcommand{\gref}[1]{{(\ref{#1})}}

\newcommand{\fmat}{G}
\newcommand{\fvec}{g}
\newcommand{\diag}[1]{\ensuremath{\mathop{\mathrm{diag}}\left( #1 \right)}}
\newcommand{\tx}{\ensuremath{\tilde{x}}}
\newcommand{\intab}[2]{\ensuremath{\int\limits_{#1}^{#2}}}
 \newcommand{\half}{\ensuremath{\frac{1}{2}}}
 \newcommand{\vecop}[1]{\ensuremath{\mathop{\mathrm{vec}}\left( #1 \right)}}
\newcommand{\beq}{\begin{equation}}
\newcommand{\eeq}{\end{equation}}

\theoremstyle{definition}
\newtheorem{theorem}{Theorem}
\newtheorem{lemma}[theorem]{Lemma}
\newtheorem{corollary}[theorem]{Corollary}

\newtheorem{remark}[theorem]{Remark}

\begin{document}
\maketitle
\begin{abstract}
The convergence of GMRES for solving linear systems can be influenced heavily by the structure of the right hand side. Within the solution of
eigenvalue problems via inverse iteration or subspace iteration, the right hand side is generally related to an approximate invariant subspace of the linear
system. We give detailed and new bounds on (block) GMRES that take the special behavior of the right hand side into account and explain the initial sharp
decrease of the GMRES residual. The bounds motivate the use of specific preconditioners for these  eigenvalue problems, e.g. tuned and polynomial
preconditioners, as we describe.  The numerical results show that the new (block) GMRES bounds are much sharper than conventional bounds and that preconditioned
subspace iteration with either a tuned or polynomial preconditioner should be used in practice.
\end{abstract}
\section{Introduction}

We consider the convergence of GMRES~\cite{SaadSchultz86} for solving sequences of linear systems of the form 
\beq
\label{eq:linsys1}
Bx_i = y_i,
\eeq
where $B = A-\sigma I$, $A\in\mathbb{C}^{n\times n}$ and  $\sigma$ is a fixed or variable scalar shift. 
Throughout, we assume that  $A$ is diagonalizable 
with eigenvalues $\eiga_j$, $j = 1,\dotsc, n$, so that  $B = A - \sigma I$ has eigenvalues $\eigb_j = \eiga_j - \sigma$, $j = 1,\dotsc, n$. 
Thus, we can write $B = Z\Eigb Z^{-1}$, where 
$\Eigb = \diag{\eigb_1, \dotsc, \eigb_n}\in\C^{n\times n}$ 
and $Z = [z_1, \ z_2, \ \dotsc, \ z_n]$ is a matrix of eigenvectors. 
Without loss of generality, we let $y_i$ be an approximation of the eigenvector $z_1$ (with both vectors normalized to have unit norm). 
Our main aim is to describe accurately, using simple bounds, the convergence of GMRES when the right-hand side is an approximate eigenvector; to clearly explain why this convergence behavior is lost for many preconditioners; and how it can be recovered by choosing or modifying the preconditioner appropriately. 

Systems of the form~\eqref{eq:linsys1} arise, for example, when computing an eigenvector and corresponding eigenvalue of a matrix $A\in\mathbb{C}^{n\times n}$ using inverse
iteration (see Algorithm~\ref{algo:inverse_iter}). In this case the shift is chosen to be close to an eigenvalue of $A$ and the approximate eigenvector $y_i$ is
the $i$th iterate of inverse iteration. 
If the system \gref{eq:linsys1} is solved exactly (e.g.\ using a
direct solver) one can show that, under certain generic assumptions, the right-hand side $y_i$ converges at least linearly to an eigenvector of $A$ (see, e.g.
\cite{PetW93,Ipsen97} and references therein). 
However, for very large matrix sizes $n$, direct solvers for \gref{eq:linsys1} are infeasible and an iterative method is used to solve  \gref{eq:linsys1}  inexactly. 
In this case, one can show that if the solve tolerance is chosen appropriately (e.g.\ it is reduced proportionally to the eigenvalue residual) we still obtain the same convergence rates for this inexact version of inverse iteration \cite{FrSp05b,FrSp06a}. 

\begin{algorithm}[H]
\DontPrintSemicolon
  \KwIn{$x_0$, $i_{\text{max}}$}
\For{$i = 0,\dotsc,i_{\text{max}}-1$}{
Choose $\sigma$ and $\tau_i$\;
Find $y_i$ such that $\|(A-\sigma I) y_i - x_i\| \le \tau_i$\;
Set $x_{i+1} = y_i/\|y_i\|$ and $\lambda_{i+1} = x_{i+1}^TAx_{i+1}$\;
Evaluate $r_{i+1} = (A - \lambda_{i+1}I)x_{i+1}$ and test for convergence\;
}
\KwOut{$x_{i_{\text{max}}}$, $\lambda_{i_{\text{max}}}$\;}
\label{algo:inverse_iter}
  \caption{Inexact inverse iteration}
\end{algorithm}

Here we are concerned with the the convergence behavior of iterative methods for solving \gref{eq:linsys1}. 
It is well known that when the right-hand side is an approximate 
eigenvector of $B$, Krylov subspace methods converge faster 
than for generic right-hand sides. This was first observed for 
Hermitian positive matrices by Simoncini and Eld\'en~\cite{SimonElden02}, who considered linear solves in Rayleigh quotient iteration.  
Further results in this direction include convergence
bounds in~\cite{BeGrSp04} for MINRES (see~\cite{PaiS75}) applied within inverse iteration for symmetric 
 matrices, and  
Theorem 2.2 in~\cite{FrSp07b}, which characterizes GMRES convergence 
for non-Hermitian generalized eigenvalue problems using the 
Schur factorization. 
For more general systems that do not necessarily come from 
eigensolvers, research into GMRES convergence has
also  highlighted the influence of  the 
right-hand side~\cite{APS98,DTM14}, and  
its representation in the eigenvector basis~\cite{MDT15, TPPW14}.

In this work, we show that if $A$ is diagonalizable  the  convergence theory developed in \cite{TPPW14} yields an insightful explanation for the GMRES behavior
in inverse iteration
 with different
types of preconditioned inner solves. Moreover, we explain why the GMRES residual often decreases sharply in the first iteration \cite{ElmXue11}. A more detailed description of this phenomenon is given in Section~\ref{sec:initdecrease}.
Regarding the preconditioned situation, it is well known that so
called tuned preconditioners can significantly reduce the number of required GMRES steps. We show that using polynomial preconditioners for the inner solves can
be similarly beneficial.

The rest of this paper is structured as follows. In Section~\ref{sec:gmres_bound} we present the GMRES bounds from~\cite{TPPW14} and show why the GMRES residual
norm often has a rapid initial decrease. In Section~\ref{sec:pre_gmres} we apply these bounds to  GMRES for inverse iteration  without  preconditioning, and 
with standard, tuned and  polynomial preconditioners. Extensions to generalized eigenvalue problems and block GMRES are discussed in                            
 Section~\ref{sec:extensions}, while numerical experiments can be found in Section~\ref{sec:num_ex}.

Throughout, 
 $k$ denotes the iteration count for the GMRES algorithm and $i$ is the iteration count for the outer iteration (which is inverse iteration here). In our analysis $j$ will denote the index of eigenvalues and eigenvectors, e.g. $Bz_j = \lambda_j z_j$ and is not to be mixed up with the iteration indices. 

\section{Analysis of the GMRES convergence bound}\label{sec:gmres_bound}

 In this section we apply results from~\cite{TPPW14} to the system 
in~\eqref{eq:linsys1} to give a lower bound on the number of GMRES 
iterations required to achieve a given residual reduction 
and explain why GMRES residuals rapidly 
decrease in the first iteration when $y_i \approx z_1$, i.e., when 
$y_i$ is a good approximation of an eigenvector of $A$. 
The GMRES convergence bounds we use to achieve this are  
related to those in~\cite{BeGrSp04} for Hermitian problems and to more 
general results in \cite{FrSp07b}
for non-Hermitian generalized eigenvalue problems. 

We begin by stating the relevant results from~\cite{TPPW14}. 
The key idea of these bounds  
is to express $\yy$ in the eigenvector basis. 
Accordingly we introduce 
\[\ww^{[i]} = Z^{-1}\yy/\|\yy\|_2 = \begin{bmatrix}\ww_1^{[i]} & \ww_2^{[i]} & \dotsc & \ww_n^{[i]}\end{bmatrix}^T = \begin{bmatrix} w_1^{[i]} & {\wt}^{[i]}\end{bmatrix}^T,\] 
where ${\wt}^{[i]} \in \C^{n-1}$. 
If $\yy\approx z_1$ it is reasonable to expect that $|\ww_1^{[i]}| \gg\|{\wt}^{[i]}\|_2$ 
and this is observed in practice (see results in corresponding columns of Tables~\ref{t:cdde:U}--\ref{t:all}). 
Note that in the rest of this section, we drop the outer iteration index $i$ from $w$ and 
related quantities for clarity. 
For simplicity, let us assume that the initial guess for GMRES is the zero vector, 
so that $r_0 = \yy$. 
Since $\yy$ is normalized at every step to have unit norm,
 $\|r_0\| = \|\yy\|=1$ $\forall i$.
 
We are now in a position to recall Theorem~2.2 from~\cite{TPPW14}, which states that 
\begin{equation}\label{eq:2.5}
\|r_k\|_2 \le \|Z\|_2 \minq\| \WW q(\Eigb)e\|_2 
= \|Z\|_2 \minq \left(\sum_{j = 1}^n |\ww_j q(\eigb_j)|^2 \right)^{\frac{1}{2}},
\end{equation}
where $\WW = \diag{\ww_1,\dotsc,\ww_n}$. 
This bound highlights  
the role of the eigenvalues of $B$ weighted by the corresponding component of $\ww$
in the GMRES convergence. 
In particular, whenever $\|\wt\|_2$ is small enough, 
$|\ww_j q(\eigb_j)| \ll |\ww_1 q(\eigb_1)|$, $j = 2, \dotsc, n$ and the 
minimizing polynomial will focus first on $|w_1 q(\eigb_1)|$. 
Thus, fast convergence of GMRES is linked both with the spectrum of $B$ and with 
the quality of the eigenvector approximation $\yy$. 

A slightly different bound is obtained if, similarly to~\cite{BeGrSp04,FrSp07b},  
we replace the minimizing polynomial
in~\eqref{eq:2.5} by $\tilde q_k(\eigb) = (1-\eigb/\eigb_1)q_{k-1}(\eigb)$,
where $q_{k-1}$ is a polynomial of degree $k-1$ for which $q_{k-1}(0) = 1$.
This is subject of Theorem 2.3 in~\cite{TPPW14}, which states that for $k>1$, 
\begin{subequations}
\label{eq:2.6}
\begin{align}
\|r_k\|_2
& \le  \|Z\|_2  \min_{\substack{q\in\Pi_{k-1},\\q(0)=1}} 
\left(\sum_{j = 2}^n |\tilde \ww_j q(\eigb_j)|^2 \right)^{\frac{1}{2}} \label{eq:2.6a}\\
& \le \|Z\|_2 \|\tilde \ww\|_2 \min_{\substack{q\in\Pi_{k-1},\\q(0)=1}} 
\max_{\eigb\in\{\eigb_2,\dotsc,\eigb_n\}}|q(\eigb)|, \label{eq:2.6b}
\end{align}
\end{subequations}
where $\tilde \ww_j= \ww_j(1-\eigb_j/\eigb_1)$, $j=2,\ldots,n$. Starting from \eqref{eq:2.5}, 
a similar approach to the proof of~\eqref{eq:2.6} shows that, additionally, 
\begin{equation} \label{eq:2.6k1}
\|r_1\|_2 \le \|Z\|_2 \left(\sum_{j = 2}^n |\tilde \ww_j |^2 \right)^{\frac{1}{2}} 
= \|Z\|_2 \|\tilde \ww\|_2.
\end{equation}
In contrast to~\eqref{eq:2.5}, neither~\eqref{eq:2.6} nor \eqref{eq:2.6k1} 
 involves $w_1$ and so they emphasize the fact
 that the relative residuals may be small when 
$\|\tilde w\|_2$ is small. 
(Note that we may always normalize $\|Z\|_2$, but that this affects $\ww$, 
and hence $\|\tilde \ww\|_2$).

\subsection{Inner iterations vs outer iterations}
In~\cite{FrSp07b} it was shown that the term  
\begin{equation}\label{eq:min_max}
\min_{\substack{q\in\Pi_{k-1},\\q(0)=1}} 
\max_{\eigb\in\{\eigb_2,\dotsc,\eigb_n\}}|q(\eigb)|
\end{equation}
can often be bounded 
by an expression of the form 
\begin{equation}\label{eq:power}
S(1/C)^{k-1}
\end{equation} (by, e.g.\ \cite[Proposition 2.3]{FrSp07b}) where $C$ and $S$ depend on the spectrum of $B$. The authors of~\cite{FrSp07b} used this bound to show that  
the number of GMRES iterations required to reach a desired tolerance,
 i.e.,  to satisfy $\|r_k\|_2/\|r_0\|_2 \le \tau$, can also be bounded.
 Combining~\eqref{eq:2.6} and~\eqref{eq:power} gives us a simple alternative lower 
bound on the number of GMRES iterations: 
\begin{equation}
\label{eq:iter_bound}
 k \ge 1 + \frac{1}{\log(C)}\left[\log\left(S\right) + 
\log\left(\frac{\|Z\|_2\|\tilde \ww\|_2}{\tau}\right) \right] .
\end{equation}

Since both $C$ and $S$ depend only on the spectrum of $A$ and the shift $\sigma$, 
they are independent of the inverse 
iteration step $i$. Thus, if we can bound  ${\|Z\|_2\|\tilde \ww\|_2}/{\tau}$ independently of the 
inverse iteration step then the above bound on $k$ suggests that 
 the number of GMRES iterations should not increase as 
 inverse iteration proceeds. We will observe that this occurs if GMRES is applied either without a preconditioner, or with specially chosen preconditioners (see Tables~\ref{t:cdde:U} and \ref{t:cdde:T}).

\subsection{Initial decrease in GMRES residuals}
\label{sec:initdecrease}
Another phenomenon that often arises when solving linear 
systems with approximate eigenvectors is that $\|r_1\|_2 \ll \|r_0\|_2$
(see e.g. Figure~\ref{f:cdde1_residuals}). 
The bound~\eqref{eq:2.6} shows clearly why this occurs. 

First note that
\[|\tilde \ww_j| = \left| 1-\frac{\eigb_j}{\eigb_1}\right| |\ww_j| = \frac{|\eiga_1-\eiga_j|}{|\eiga_1-\sigma|}|\ww_j| 
\le |\ww_j| \max_{j\in[2,\dotsc,n]} \frac{|\eiga_1-\eiga_j|}{|\eiga_1-\sigma|} = C_1|\ww_j|,\]
where the constant $C_1$ depends only on the spectrum of $A$ and the shift $\sigma$.
Accordingly, $\|\tilde \ww\|_2 \le C_1 \|\wt\|$, and~\eqref{eq:2.6k1} becomes 
\begin{equation}
\label{eq:r1}
\|r_1\|_2 \le C_2 \|\wt\|_2,
\end{equation}
where $C_2 =  \|Z\|_2 \max_{j\in[2,\dotsc,n]} {|\lambda_1-\lambda_j|}/{|\lambda_1-\sigma|}$. 

However, the values of $w_j$, $j=2,\ldots n$ are very small, and indeed are zero if the right-hand side is an exact eigenvector. Hence, in the latter
stages of inverse iteration there is a sharp decrease in the bound for the relative GMRES residual norms. Of course, this is only true if the right-hand side is
an approximate eigenvector of $B$, i.e.\ in the unpreconditioned case, or with specially chosen preconditioners. An example of this phenomenon is given in
Figure~\ref{f:cdde1_residuals}.
\section{Convergence bounds for preconditioned GMRES}\label{sec:pre_gmres}

Having considered a general analysis for GMRES convergence that involves the right-hand side vector, we now investigate what this analysis tells us about  solving the linear systems in inverse iteration, both with and without preconditioning.

\subsection{No preconditioner}

When GMRES is applied to \gref{eq:linsys1} without a preconditioner, the coefficient matrix is $B = A - \sigma I$, $A\in\C^{n\times n}$ 
where, assuming $A$ is diagonalizable, $B = Z\Lambda Z^{-1}$ with $\Lambda\in\C^{n\times n}$. 
The right-hand side $y_i$ is an approximation of $z_1$ (and both vectors are normalized to have unit norm). Thus it is possible to write 
\beq
\label{eq:decomp}
y_i = \alpha_i z_1 + Z_2 p_i,
\eeq
where $\alpha_i\in\mathbb{C}$, $Z_2 = [z_2,\ldots,z_n]\in \mathbb{C}^{n\times n-1}$ and $p_i\in\mathbb{C}^{n-1}$. We assume that $\alpha_i$ and $p_i$ are chosen such that $\|y_i\|=1$ $\forall i$. The scalar $\alpha_i$ measures the deviation of $y_i$ from $z_1$ and, under generic assumptions for exact solves, inverse iteration converges, that is $\alpha_i\rightarrow 1$ and $\|p_i\|\rightarrow 0$ as $i\rightarrow\infty$. 

The bound in \gref{eq:2.5}, and in particular the vector $w$, depends on the outer iteration~$i$. Using \gref{eq:decomp} we obtain
\begin{eqnarray*}
w^{[i]} &=& Z^{-1}y_i= Z^{-1}(\alpha_i z_1 + Z_2 p_i)\\
&=& \alpha_i e_1 + E_{n-1}p_i,
\end{eqnarray*}
where $e_1\in\mathbb{R}^n$ and $E_{n-1}\in\mathbb{R}^{n,n-1}$ contain the first and the last $n-1$ columns of the identity matrix, respectively. Recall that $w^{[i]}$ denotes the vector $w$ at outer iteration $i$. With $\alpha_i\rightarrow 1$ and $\|p_i\|\rightarrow 0$ as $i\rightarrow\infty$ for a convergent outer iteration we obtain 
$w^{[i]}\rightarrow e_1$. Hence in the limit $q(\eigb)$ only needs to be minimized over $\eigb_1$.

Noting that  
$\|Z\|_2$ can be normalized,  $\|w^{[i]}\|_2\rightarrow 1$ 
and $\|\tilde{w}^{[i]}\|_2 \rightarrow 0$, we see that 
the terms $\|Z\|_2\|w^{[i]}\|_2$ and $\|Z\|_2\|\tilde{w}^{[i]}\|_2$ in~\eqref{eq:2.5} and ~\eqref{eq:2.6} can be 
bounded by an arbitrarily small constant for large enough $i$. Therefore, as the outer iteration progresses, small  relative GMRES residual norms are rapidly obtained. 
Experimentally we see that the number of inner iterations is bounded
because $\|\tilde{w}^{[i]}\|_2\rightarrow 0$ at a rate proportional to the decrease of $\tau$. 
This is reflected by the lower bound in~\eqref{eq:iter_bound}, which is constant if  $\|\tilde{w}^{[i]}\|_2 \propto \tau$  (cf. Lemma 3.11 in~\cite{frei07}). 
As the outer iterations progress (i.e.\ for larger $i$) there is an initial decrease in the relative GMRES residual norms, as suggested by~\eqref{eq:r1},  
because $\|\tilde{w}^{[i]}\|_2 \approx 0$. 

\subsection{Standard preconditioner}

We now investigate the changes that occur when a preconditioner $P$ is introduced.  
Usually GMRES is used
with a preconditioner  to cluster the eigenvalues of the system matrix. The aim of this is to  reduce the bound on the term~\eqref{eq:min_max}
and hence improve the convergence bound~\eqref{eq:2.6}. We will see that this usually comes at the expense of increasing $\|\tilde \ww\|_2$, so that the number of GMRES iterations actually grows as inverse iteration proceeds (although this number is generally still lower than the number of GMRES iterations needed without a preconditioner). 

 Without loss of generality we consider a right
preconditioner for \gref{eq:linsys1}, that is 
\beq
BP^{-1}\tilde{x}_i = y_i, 
\eeq
where $ P^{-1}\tilde{x}_i = x_i$.
Although the eigenvalues of $BP^{-1}$ may be clustered, the eigenvectors will typically differ from those of $B$. Thus, although $y_i$ is an approximate
eigenvector of $B$, it is not usually an approximate eigenvector of the coefficient matrix $BP^{-1}$. This causes  the number of GMRES iterations to increase as
inverse iteration progresses,  as we now show.

In general we have $BP^{-1}\neq ZDZ^{-1}$ (with $D$ a diagonal matrix), that is,  $BP^{-1}$ does not have the same eigenvectors as $B$. Hence, with $BP^{-1}= \bar{Z}D\bar{Z}^{-1}$ and using the decomposition of $y_i$ we obtain
\[
w^{[i]} = \bar{Z}^{-1}y_i= \alpha_i \bar{Z}^{-1} z_1 + \bar{Z}^{-1}Z_2 p_i\rightarrow \alpha_i \bar{Z}^{-1} z_1,
\]
as $i\rightarrow\infty$ since $\|p_i\|\rightarrow 0$, but $\alpha_i \bar{Z}^{-1} z_1\neq e_1$ in general. 

Hence, there is no reason for $\|{\wt}^{[i]}\|_2$ to be small. Instead, $\|{\wt}^{[i]}\|_2\rightarrow c$ for some constant $c$ as $i$ increases (see
corresponding column in Table~\ref{t:cdde:U}) and we observe that for some problems $\|{\wt}^{[i]}\|_2 > |w_1^{[i]}|$, i.e., $c>1$.
Since $\|{\wt}^{[i]}\|_2$ does not approach $0$, but $\tau$ decreases as the outer iteration progresses, 
we would expect from~\eqref{eq:iter_bound} that the number of inner iterations increases with increasing~$i$, and indeed this is what our experiments in Section~\ref{sec:num_ex} show. 
Moreover, we find that $\|r_1\|_2$ is not significantly smaller than $\|r_0\|_2$ at later outer iterations. 
However, for large enough $i$ the terms $\|Z\|_2\|w^{[i]}\|_2$ or $\|Z\|_2\|\tilde{w}^{[i]}\|_2$ can be bounded by constants that are are still small 
relative to $\kappa_2(Z) = \|Z\|_2 \|Z^{-1}\|_2$, and the bounds~\eqref{eq:2.5} and \eqref{eq:2.6} can still capture GMRES convergence behavior, especially at later outer iterations (see Figure~\ref{f:cdde1_residuals}). 

To recover bounded numbers of inner iterations when preconditioning we must ensure that the preconditioner preserves the relevant direction given by the right-hand side. For this, we may construct a preconditioner such that 
\begin{enumerate}[(a)]
\item $BP_i^{-1}= \bar{Z}_iD_i\bar{Z}_i^{-1}$ such that (in the limit for large $i$) $\bar{Z}_i^{-1} z_1 \rightarrow e_1$, or
\item $BP^{-1}$ has the same eigenvectors as $B$.
\end{enumerate}

The next two subsections show how these preconditioners can be constructed.

\subsection{Tuned preconditioner}

Assume we have an a preconditioner $P$ such that $BP^{-1}$ is diagonalizable. Then $BP^{-1} = Z\Lambda Z^{-1}P^{-1} = \bar{Z}D\bar{Z}^{-1}$, where $D$ is a diagonal matrix. Since we want to recover bounded GMRES iteration counts, we want $y_i$ to be an approximate eigenvector of $BP^{-1}$. To achieve this, it is not necessary that $BP^{-1}$ has all the same eigenvectors as $B$, but it should have the same eigenvector in the direction of the sought eigenvector $z_1$, that is, we want to enforce $\bar{Z}^{-1} z_1 = e_1$. 

If we enforce this condition,  we obtain
\[
Z\Lambda Z^{-1}P^{-1}z_1 = \bar{Z}D\bar{Z}^{-1}z_1 = \bar{Z}De_1 = d_{11} z_1.
\]
Since $\Lambda = \text{diag}(\lambda_1,\ldots,\lambda_n)=\text{diag}(\gamma_1-\sigma,\ldots,\gamma_n-\sigma)$, and assuming that $\gamma_j\neq\sigma$, $\forall j$, the above requirement shows that $P^{-1}$ needs to satisfy
\beq
\label{eq:tuning}
P^{-1}z_1 = \frac{d_{11}}{\lambda_1} z_1.
\eeq
Hence, choosing a preconditioner $P$ such that \gref{eq:tuning} holds, or equivalently
\[
Pz_1 = \frac{\lambda_1}{d_{11}} z_1,
\]
yields $\bar{Z}^{-1} z_1 = e_1$. Note that $d_{11}\neq 0$ (otherwise $\lambda_j=0$ and $\gamma_j=\sigma$ for some $j$, a case we excluded). 

Clearly, \gref{eq:tuning} is infeasible as a preconditioner, since $z_1$, $\lambda_1$ and $d_{11}$ are unknown during the iteration. Hence, we propose to use a preconditioner $P_i$ at each iteration $i$ that satisfies
\beq
\label{eq:tunedpract}
P_i y_i = \lambda^{(i)}y_i =  (\gamma^{(i)}-\sigma) y_i,
\eeq
where $\gamma^{(i)}$ is an eigenvalue approximation for $\gamma_1$; it can be obtained as part of the the inverse iteration algorithm.
\begin{remark} Instead of \gref{eq:tunedpract} one could use 
\begin{itemize}
\item $P_i y_i = y_i$, since the direction is important for the preconditioner;
\item $P_i y_i = B y_i$, since $B y_i = (\gamma^{(i)}-\sigma) y_i + r_i = \lambda^{(i)} y_i + r_i$ with $r_i = (A-\gamma^{(i)}I)y_i$ the eigenvalue residual, and in the limit \gref{eq:tunedpract} holds;
\item $P_i y_i = A y_i$, since $A y_i = \gamma^{(i)} y_i + r_i$, which, in the limit, lies in the same direction.
\end{itemize}
The action of $P_i^{-1}$ on a vector $v$ can easily be obtained as an update of  $P^{-1}v$ using the Sherman-Morrison formula, that is
\[
P_i^{-1}v = P^{-1}v -\frac{P^{-1}y_i -(\lambda^{(i)})^{-1} y_i}{y_i^T P^{-1}y_i} y_i^T P^{-1}v.
\]
This can be implemented using only one extra linear solve with $P$ (to compute $P^{-1}y_i$) per outer iteration. Note that we assume that $\lambda^{(i)}\neq 0$.
\end{remark}

\begin{theorem}
Consider inverse iteration, where at each inner iteration a preconditioned system of the form 
$BP_i^{-1}\tilde{x}_i = y_i$ is solved with $P_i^{-1}\tilde{x}_i = x_i$, and $P_i$ satisfies  \gref{eq:tunedpract}. Assume that $BP_i^{-1}$ is diagonalizable $\forall i$, that is $BP_i^{-1} = Z_iD_iZ_i^{-1}$ with $D_i$ diagonal. Further, assume that $B$ is nonsingular, that is $\lambda^{(i)}=\gamma^{(i)}-\sigma\neq 0$ $\forall i$. Then
\begin{enumerate}[(i)]
\item $BP_i^{-1}y_i = y_i + \frac{r_i}{\gamma^{(i)}-\sigma}$, where $r_i = (A-\gamma^{(i)}I)y_i$ is the eigenvalue residual (that is, in the limit $BP_i^{-1}$ has an eigenvalue at $1$);
\item $w^{[i]}\rightarrow e_1$ as $i\rightarrow\infty$.
\end{enumerate}
\end{theorem}
\begin{proof}
\begin{enumerate}[(i)]
\item Using the definition of $P_i^{-1}$ we obtain 
\[
BP_i^{-1}y_i = (\gamma^{(i)}-\sigma)^{-1} B y_i = y_i + \frac{r_i}{\gamma^{(i)}-\sigma},
\]
where we have used the fact that $By_i=(A-\sigma I)y_i = (\gamma^{(i)}-\sigma)y_i+r_i$.
\item We need to compute $w^{[i]} = Z_i^{-1}y_i$ as $i\rightarrow\infty$. Consider the eigendecomposition $BP_i^{-1} = Z_iD_iZ_i^{-1}$ as $i\rightarrow\infty$. Then $y_i\rightarrow z_1$ and, using part (i) we have  
\[
BP_i^{-1}y_i\rightarrow z_1,
\]
since $r_i\rightarrow 0$.
Therefore, in the limit $Z_i \rightarrow [z_1, \tilde{Z}_2]$, as $i\rightarrow\infty$, for some $\tilde{Z}_2$ and hence,
\[
w^{[i]} = Z_i^{-1}y_i \rightarrow  [z_1, \tilde{Z}_2]^{-1} z_1 = e_1.
\] 
\end{enumerate}
\end{proof}

We obtain $w^{[i]} \rightarrow e_1$. Since 
 $\|Z\|_2$ can be normalized, $\|w^{[i]}\|_2\rightarrow 1$, 
and $\|\tilde{w}^{[i]}\|_2\rightarrow 0$, the terms $\|Z\|_2\|w^{[i]}\|_2$ and $\|Z\|_2\|\tilde{w}^{[i]}\|_2$
in~\eqref{eq:2.5} and \eqref{eq:2.6} can be bounded by small constants for large enough $i$.  
Experimentally we see that the number of inner iterations is bounded 
because $\|w^{[i]}\|_2\rightarrow 0$ at a rate proportional to $\tau$
(cf.\ Theorem~6.22 in~\cite{frei07}). Additionally, as the outer iterations progress there is an initial decrease 
in the relative GMRES residual norms, that is, $\|r_1\|_2 \ll \|r_0\|_2$,  because 
$\|\tilde{w}^{[i]}\|_2 \approx 0$.

\subsection{Polynomial preconditioner}\label{ssec:polyprec}
As an alternative to tuning we can construct a preconditioner such that $BP^{-1}$ has the same eigenvectors as $B$. 
If we choose 
\[
P^{-1} = p(B),
\]
where $p(B)$ is a polynomial of $B$, then $p(B) = Zp(\Lambda) Z^{-1}$ and $BP^{-1} = Z\Lambda p(\Lambda) Z^{-1}$ has the same eigenvectors as $B$. In this case  the theory
for the unpreconditioned case applies.

We consider the (right) polynomially preconditioned system $ Bp(B)\tx_i=y_i$, where $x_i=p(B)\tx_i$ and 
\begin{align*}
p(z)=\sum\limits_{h=0}^{d}\mu_hz^h\in\Pi_d.
\end{align*}
Since $p(B)=Zp(\Lambda)Z^{-1}$, $y_i$ is still an approximate eigenvector of the preconditioned coefficient matrix $Bp(B)$ and it holds that 
\begin{equation*}
 \|r_k\|_2\leq \|Z\|_2\min\limits_{\stackrel{q_k\in\Pi_k}{q_k(0)=1}}
 \left(\sum_{j = 1}^n\vert w_jq_k(\lambda_jp(\lambda_j))\vert^2\right)^{\half}.
 \end{equation*}
Hence, all the weights $w$ of the unpreconditioned version are preserved. 
Typically, $p(B)$ is chosen such that
\begin{align*}
 B(p(B))\approx I\quad\text{or}\quad g(B):=I-Bp(B)\approx 0.
\end{align*}
In the latter choice $g\in\Pi_{d+1}$ is the residual polynomial, which can be written as
\begin{align*}
 g(z)=\prod\limits_{h=1}^{d+1}(1-\nu_h z)=1-\sum\limits_{h=1}^{d+1}\mu_{h-1}z^{h}.
\end{align*}
There are different strategies to choose the polynomial coefficients $\mu_h$, which 
can be determined from the $\nu_h$ recursively via
\begin{align*}
 g_{h}(z)=g_{h-1}(z)-\nu_hzg_{h-1}(z),\quad h=1,\ldots,d+1,~g_0(z)=1, 
\end{align*}
see, e.g.\ \cite{AhmSG12} for more details.
A thorough discussion of various choices for the polynomial coefficients is beyond the scope of this study and we restrict ourselves to a few selected, often
used approaches. More information on the choice of the polynomial can be found, e.g., in \cite{Freund1990,AshMO92,Gij95,AhmSG12,templatesLS,LiuMW15} and the
references therein.

A basic common choice \cite{AhmSG12} is to take the $\nu_h$ as reciprocal Chebychev nodes with respect to an interval $[a,b]$:
\begin{align}\label{polycoeffcheb}
 \phi_h=\frac{2h-1}{2(d+1)},\quad \nu_h=2\left(b+a-(b-a)\cos(\pi\phi_h)\right)^{-1},\quad h=1,\ldots,d+1.
\end{align}
Here, $a,b$ denote the smallest and largest eigenvalues (or approximations thereof) of $\Lambda(B)\subset\R$. For complex spectra, $a,b$ are the foci of the
ellipse enclosing $\Lambda(B)$. The values $a,b$ can be obtained from approximate eigenvalues of $A$ or $B$, e.g.\ by employing 
\texttt{eigs} with the 'SM' and 'LM' flags, or using the Ritz values generated by the Arnoldi process. Alternatively one can compute a
very small number of
(harmonic)  Ritz values \cite{Gij95,AhmSG12}, which can be generated by a few steps of an (harmonic) Arnoldi process.
When using this approach, it is important to ensure that $a,b$ have the same sign but in some of our examples this not the case and either $a$ or $b$ lie a little bit on the wrong side
of 
the origin. In these cases we slightly shifted $a,b$, e.g.\ if $b>0$ but $a<0$ is small we set $b\leftarrow b-2a$, $a\leftarrow -a$.

For complex spectra  a more sophisticated approach is proposed in \cite{Gij95}. The key idea is to consider the norm 
\begin{align}\label{polycoeffmin}
 \|g\|^2=\frac{1}{L}\intab{\Gamma}{}g(z)\overline{g(z)}\vert\mathrm{d}z\vert
\end{align}
induced by the scalar product $\langle f,g\rangle=\frac{1}{L}\intab{\Gamma}{}f(z)\overline{g(z)}\vert\mathrm{d}z\vert$.
Here $\Gamma$ is a piecewise linear contour approximating the shape of the spectrum of $B$ and $L$ is its arc length.
The coefficients $\mu_i$ in the polynomial are determined such that the above norm is minimized in a least-squares sense. 
Using (harmonic) Ritz values allows us to easily generate the piecewise linear contour, e.g.  using the convex hull of the Ritz values.
However, this approach also allows us to approximate the spectrum by non-convex objects such as a boomerang-shaped region, which can be beneficial in some cases.
Note that if $B$ is real, it is sufficient to incorporate only the part of $\Gamma$ with nonnegative imaginary parts. More details on the implementation of this
strategy are given in \cite{Gij95}. 

A polynomial preconditioner preserves the eigenvector basis and hence appears to be an ideal preconditioner to use within GMRES for linear systems arising
within inverse iteration (using the theory in this paper). However, we would like to note that 
a limitation of good polynomial preconditioners is that we require knowledge about the location of the spectrum of $B$.
Since we are actually seeking an
eigenvalue, the necessary information to construct a good polynomial preconditioner might be hard to obtain. We already mentioned the use of spectral estimates, which can be obtained by, e.g., (harmonic) Arnoldi processes. 
Since we are using inexact inverse iteration with GMRES as inner solver, one
 strategy deserving further study is to use the  upper Hessenberg matrix generated at  
outer iteration step $i-1$ to acquire spectral estimates for outer iteration step $i$. In a similar way, the generated basis matrices of the Krylov subspaces can be used to acquire the coefficients
by solving a least square problem along the lines of \cite{LiuMW15}.

For completeness we mention a third polynomial preconditioning strategy. Since iterative Krylov methods for linear systems work implicitly with polynomials
of $B$, we could in principle use any of these  methods as a preconditioner for the inner iteration. In other words, we could use (a small number of steps of) a Krylov
method within a Krylov method, which is GMRES here. However, since Krylov methods also depend on the right-hand side,  which determines the first basis vector in
the Krylov subspace,
the preconditioner will change with every iteration of GMRES in our consideration. Therefore, a flexible preconditioned Krylov method, such as flexible
GMRES~\cite{Saa93} must be applied, which is beyond the scope of this study. Similar approaches can be found in, e.g.  \cite{GijSZ15,DuSZ15,BG15}. The
advantage of this strategy is that one does not have to worry about the generation of the polynomial coefficients.
\section{Extensions and related issues} \label{sec:extensions}

In this section several extensions to inverse iteration for the standard eigenvalue problems are considered.
First we extend our analysis of GMRES within inverse iteration to generalized eigenvalue problems. Then,  we obtain bounds for block GMRES arising in inverse subspace iteration for the standard eigenvalue problem. 

\subsection{Generalized eigenvalue problems} \label{subsec:gen_eig}

Consider inverse iteration for the generalized eigenvalue 
problem $Ax = \lambda Mx$, where at least one of the matrices $A$ and $M$ is nonsymmetric. We shall also assume that the desired eigenvalue is finite in the case that 
$M$ is singular. 
For generalized eigenvalue problems, the linear systems we wish to solve in inverse iteration are of the form 
$(A - \sigma M)\xx = M\yy$, where 
 we normalize $\yy$ so that 
$\|M\yy\|_2  = 1$. The eigenvalue approximation can then be obtained by a generalization of the Rayleigh quotient, e.g.\ 
$(\xx^TM^TA\xx)/\|M\xx\|^2$.
Let us now assume that $A - \sigma M$ is diagonalizable, i.e.\ 
$A - \sigma M = Z \Lambda Z^{-1}$ where $Z$ is a matrix 
of eigenvectors and $\Lambda$ is a diagonal matrix of eigenvalues. 
We also assume without loss of generality that $\yy$ approximates 
$z_1$, the first column of $Z$.

Setting 
\[w^{[i]} = Z^{-1}\yy = \begin{bmatrix} w_1^{[i]} & {\wt}^{[i]}\end{bmatrix}^T\] 
 we would again expect that 
when $\yy$ is a good approximation of $z_1$ that $|w_1^{[i]}|\gg \|{\wt}^{[i]}\|_2$. Again, for clarity, in the rest of this section, we drop the outer
iteration index $i$ from $w$ and related quantities where appropriate. The GMRES bounds in this case becomes 

\begin{equation}\label{eq:gep_gmres}
\begin{aligned}
\|r_k\|_2 &= \minq\|q(A-\sigma M)M\yy\|_2 \\
&=\minq \|Z q(\Lambda) (Z^{-1}MZ) w\|_2 = \minq \|Z q(\Lambda) f\|_2, 
\end{aligned}
\end{equation}
where $f:= Z^{-1}MZw$. Note that if $w$ approximates the first 
unit vector, then $f$ approximates the first column of $Z^{-1}MZ$. 
In this case an analogous bound to~\eqref{eq:2.5} is
\begin{equation}\label{eq:2.5gep}
\|r_k\|_2 \le \|Z\|_2 \minq\| F q(\Lambda)e\|_2 
= \|Z\|_2 \minq \left(\sum_{j = 1}^n |f_j q(\lambda_j)|^2 \right)^{\frac{1}{2}},
\end{equation}
where $F:= \diag{f_1, \dotsc, f_n}$. 
Thus, as in Section~\ref{sec:gmres_bound}, we  have a weighted polynomial approximation. 
However,  the weights 
may all be large because $M\yy$ is not close to an eigenvector of $A - \sigma M$ in general.

Considering~\eqref{eq:2.6}, we find that an analogous bound is 
\begin{equation}\label{eq:2.6gep}
\begin{aligned}
\|r_k\|_2
&\le  \|Z\|_2  \min_{\substack{q\in\Pi_{k-1},\\q(0)=1}} 
\Big(\sum_{j = 2}^n |\tilde f_j q(\lambda_j)|^2 \Big)^{\frac{1}{2}}\\
&\le \|Z\|_2 \|\tilde f\|_2 \min_{\substack{q\in\Pi_{k-1},\\q(0)=1}} 
\max_{\lambda\in\{\lambda_2,\dotsc,\lambda_n\}}|q(\lambda)|,
\end{aligned}
\end{equation}
where $\tilde f_j = (1-\lambda_j/\lambda_1)f_j$. 
Again, there is no reason for $\|\tilde f\|_2$ to be small. 
We conclude that for generalized eigenproblems, unpreconditioned GMRES may not show a residual reduction similar to the case $M=I$.

One way to reintroduce this behavior is by using the tuned preconditioner $P_i = P(I - y_iy_i^H) + My_iy_i^H$, which satisfies $P_iy_i = My_i$ \cite{FrSp07b}. We stress again that application of the tuned preconditioner requires  only one extra matrix-vector
product $P^{-1}My_i$ per outer iteration \cite{FrSp07b}. 
We also note that in the absence of a good preconditioner, the choice $P=I$ should 
at least recover the behavior observed for  inverse iteration with unpreconditioned GMRES applied to  the standard eigenvalue problem.
It is also obvious that the nice properties of polynomial preconditioners do not hold for the generalized eigenproblem, since the weight vectors $f_j$ 
will be the same as in the unpreconditioned case.  
One could  add tuning to a polynomial preconditioner, but from numerical experience this strategy was not competitive.

\subsection{Block GMRES}
\label{sec:bgmres}
The linear system in \gref{eq:linsys1} can be generalized to a block linear system,
that is, a linear system with multiple right-hand sides of the form
\beq
\label{eq:blocklinsys1}
BX_i = Y_i,
\eeq
where $B=A-\sigma I$ as before, and $Y_i$, $X_i\in\mathbb{C}^{n\times u}$, $u\lll n$. Systems of this form arise when seeking an invariant subspace using
inverse subspace iteration~\cite{RSS06}, a block version of the inverse iteration. If these systems are solved by block-GMRES, we obtain similar bounds to  \gref{eq:2.5} (see also Theorem 2.2
in~\cite{TPPW14}), and can gain insight into the convergence behavior of block-GMRES as the inverse subspace iterations progress.

\begin{lemma}\label{lem:gen_eig}
Suppose that $B$ is diagonalizable, that is $B=Z\Lambda Z^{-1}$, and block-GMRES is used to solve the linear system with multiple right-hand sides of the form $BX=Y$,
$Y\in\mathbb{C}^{n\times u}$. The residual $Y-BX_k$ associated with the approximate solution $X_k$, obtained with $k$ iterations of block-GMRES starting with
$X_0=0$ is such that
\begin{align*}
\|Y-BX_k\|_F& \le \|Z\|_2\min_{\fmat_i\in\mathbb{C}^{u\times u}} \left\|  W+ \sum_{i=1}^k \Lambda^i W\fmat_i\right\|_F\\
& = \|Z\|_2\min_{\substack{q_{\ell,j} \in \Pi_k,\\q_{\ell,j}(0) = \delta_{\ell,j}}}   \left\|\begin{bmatrix}  \sum_{\ell=1}^u q_{\ell,1}(\Lambda)w_\ell\\
 \vdots\\
  \sum_{\ell=1}^u q_{\ell,u}(\Lambda)w_\ell\\
 \end{bmatrix}
 \right\|_2,
\end{align*}
where $ W = Z^{-1}Y\in\mathbb{C}^{n\times u}$, $w_\ell$ is the $\ell$th column of $W$, $\Pi_k$ is the space of polynomials of at most degree $k$ and
$\delta_{\ell,j}$ is the Kronecker delta. 
\end{lemma} 
\begin{proof}
Since $X_k\in\text{span}\{Y,BY,\ldots, B^{k-1}Y\}$ and block-GMRES minimizes the residual we have 
\begin{align*}
\|Y-BX_k\|_F
& \le \min_{\fmat_i\in\mathbb{C}^{u\times u}}\|Y+\sum_{i=1}^k B^i Y \fmat_i\|_F \\
&=\min_{\fmat_i\in\mathbb{C}^{u\times u}} \left\|  (I\otimes Z) \vecop{W+ \sum_{i=1}^k \Lambda^i W\fmat_i}\right\|_2\\
& \le \|Z\|_2\min_{\fmat_i\in\mathbb{C}^{u\times u}} \left\|  W+ \sum_{i=1}^k \Lambda^i W\fmat_i\right\|_F,
\end{align*}
where we have used vectorization and the properties of the Kronecker product in the second step. If we introduce the additional matrix $\fmat_0 = I_u$ then 
\[\left(W+ \sum_{i=1}^k \Lambda^i W\fmat_i\right)e_j = \sum_{i=0}^k \Lambda^i W\fmat_i e_j = \sum_{i=0}^k \Lambda^i \sum_{\ell=1}^u \fvec_{\ell j}^{({i})}w_\ell =
\sum_{\ell=1}^u q_{\ell,j}(\Lambda)w_\ell  ,\]
where $e_j$ is the $j$th unit vector, $w_\ell$ is the $\ell$th column of $W$, $\fvec_{\ell j}^{(i)}$ is the $(\ell,j)$th element of $\fmat_i$ and  $q_{\ell,j}(\lambda)
= \sum_{i=0}^k \lambda^i \fvec_{\ell j}^{(i)}$, $j,\ell = 1,\dotsc,u$.  Clearly $q_{\ell,j} \in \Pi_k$. Additionally, the condition $\fmat_0=I_u$ means that
$q_{\ell,j}(0) = \delta_{\ell,j}$. 

Thus,
\begin{align*}
 \left\|  W+ \sum_{i=1}^k \Lambda^i W\fmat_i\right\|_F
 &= \left\|\begin{bmatrix} \sum_{\ell=1}^u q_{\ell,1}(\Lambda)w_\ell, & \dotsc, & \sum_{\ell=1}^u q_{\ell,u}(\Lambda)w_\ell \end{bmatrix}  \right\|_F\\
 &= \left\|\begin{bmatrix}  \sum_{\ell=1}^u q_{\ell,1}(\Lambda)w_\ell\\
 \vdots\\
  \sum_{\ell=1}^u q_{\ell,u}(\Lambda)w_\ell\\
 \end{bmatrix}
 \right\|_2,
 \end{align*}
 from which the result follows. 
\end{proof}

A consequence of Lemma~\ref{lem:gen_eig} is the following block-GMRES bound. 
\begin{theorem}
\label{th:gen_eig}
Suppose that $B$ is diagonalizable, that is $B=Z\Lambda Z^{-1}$, and block-GMRES is used to solve the linear system with multiple right-hand sides of the form $BX=Y$,
$Y\in\mathbb{C}^{n\times u}$. The residual $Y-BX_k$ associated with the approximate solution $X_k$, obtained with $k$ iterations of block-GMRES starting with
$X_0=0$, is such that
\begin{equation}
\label{eq:boundgen_eig}
\|Y-BX_k\|_2\le \|Z\|_2 \minq \left(\sum_{\ell=1}^u\sum_{j = 1}^n |\ww_{j\ell} q(\eigb_j)|^2 \right)^{\frac{1}{2}}
\end{equation}
where $Z^{-1}Y = \WW$, and $\ww_{j\ell}$ is the $(j,\ell)$th entry of $\WW$.
\end{theorem} 
\begin{proof}
In Lemma~\ref{lem:gen_eig}, replace the minimizing polynomial $q$ by $\hat q$, where $\hat q_{\ell,j}(\lambda) \equiv 0$, $j\ne \ell$ and  
\[\hat q_{11}(\lambda) = \hat q_{22}(\lambda) = \dots = \hat q_{uu}(\lambda) \equiv \hat q(\lambda).\]
(Note that this is equivalent to choosing $\fmat_i = \fvec_i I$, $\fvec_i$ constant, $i = 1,\dotsc, k$.)
Then, 
\[\|Y-BX_k\|_F \le \|Z\|_2\minq \left\|\begin{bmatrix}  q(\Lambda)w_1\\
 \vdots\\
q(\Lambda)w_u\\
 \end{bmatrix}
 \right\|_2\!
=\!\|Z\|_2\minq\Big(\sum_{\ell=1}^u\sum_{j = 1}^n |\ww_{j\ell} q(\eigb_j)|^2\Big)^{\half}\!\!.
\]
The result follows from $\|Y-BX_k\|_2\le\|Y-BX_k\|_F$.
\end{proof}

The key idea in Lemma~\ref{lem:gen_eig} and Theorem~\ref{th:gen_eig} is, again, to express the right-hand side $Y$ in the eigenvector basis of $B$,
i.e., 
\[
W = Z^{-1}Y = \begin{bmatrix}W_1\\W_2\end{bmatrix},
\]
where $W_1\in\mathbb{C}^{u\times u}$ and $W_2\in\mathbb{C}^{(n-u)\times u}$.
We write the right-hand side $Y_i\in\mathbb{C}^{n\times u}$ in the form
\[
Y_i = Z_1 Y_i^1 + Z_2 Y_i^2,
\]
where $Z_1\in\mathbb{C}^{n\times u}$ and $Z_2\in\mathbb{C}^{n\times (n-u)}$ are the first $u$ and the last $n-u$ columns of $Z$, that is $B = [Z_1\,Z_2]\Lambda
[Z_1\,Z_2]^{-1}$, and $Y_i^1 \in\mathbb{C}^{u\times u}$ is nonsingular and $Y_i^2\in\mathbb{C}^{(n-u)\times u}$.
In the generic situation, we have that inverse
subspace iteration converges, e.g. $\text{ran}(Y_i)$ converges to  $\text{ran}(Z_1)$ and hence $\|Y_i^2\|\rightarrow 0$. Hence for inverse subspace iteration
\[
W^{[i]} = Z^{-1}Y_i = Z^{-1} ( Z_1 Y_i^1 + Z_2 Y_i^2)= E_1 Y_i^1 + E_2 Y_i^2 = \begin{bmatrix}W_1^{[i]}\\W_2^{[i]}\end{bmatrix},
\]
where $E_1$ and $E_2$ are the first $u$ and the last $n-u$ columns of the identity matrix respectively. As $\|Y_i^2\|\rightarrow 0$ we have, in the limit, 
\[
W^{[i]}\approx \begin{bmatrix}W_1^{[i]}\\0\end{bmatrix},
\]
where $W_1^{[i]}\in\mathbb{C}^{u\times u}$ and otherwise $\|W_1^{[i]}\|\gg \|W_2^{[i]}\|$, similar to the case where $u=1$ in the main part of this paper. In  light of 
Lemma~\ref{lem:gen_eig} and Theorem~\ref{th:gen_eig} this means that a lot of entries of $w_\ell^{[i]}$, the $\ell$th column of $W^{[i]}$ are small or zero. In
the limit the bound in \eqref{eq:boundgen_eig} becomes
\[
\label{eq:boundgen_eig2}
\|Y_i-BX_k\|_2\le \|Z\|_2 \minq \left(\sum_{\ell=1}^u\sum_{j = 1}^u |\ww_{j\ell}^{[i]} q(\eigb_j)|^2 \right)^{\frac{1}{2}}
\]
and hence the minimizing polynomial will focus on minimizing over the relevant sought spectrum, e.g. $q(\lambda_1)\,\ldots,q(\lambda_u)$.

This property is violated when a preconditioner is applied in block GMRES, but can be overcome, in a similar way to the case $u=1$, by a tuned preconditioner,
e.g. by using a preconditioner $P_i$ which satisfies $P_iY_i = BY_i$, see~\cite{RSS06} for details.

It is also possible to use block GMRES within inverse subspace iteration for generalized eigenvalue problems. In this case the linear system that must be solved is of the form $(A-\sigma M)X_i = MY_i$, where again $X_i,Y_i\in \mathbb{C}^{n\times u}$. Assuming that $A-\sigma M = Z\Lambda Z^{-1}$ is diagonalizable, our results carry over to this case, and we obtain the following theorem. 

\begin{corollary}
\label{th:gen_eig_blk}
Suppose that $A-\sigma M$ is diagonalizable, that is $A-\sigma M=Z\Lambda Z^{-1}$, and block-GMRES is used to solve the linear system with multiple right-hand sides of the form $(A-\sigma M)X=MY$,
$Y\in\mathbb{C}^{n\times u}$. The residual $MY-(A-\sigma M)X_k$ associated with the approximate solution $X_k$, obtained with $k$ iterations of block-GMRES starting with
$X_0=0$, is such that
\begin{equation}
\label{eq:boundgen_eig_gen}
\|MY-(A-\sigma M)X_k\|_2\le \|Z\|_2 \minq \left(\sum_{\ell=1}^u\sum_{j = 1}^n |h_{j\ell} q(\eigb_j)|^2 \right)^{\frac{1}{2}}
\end{equation}
where $Z^{-1}MY = H$, and $h_{j\ell}$ is the $(j,\ell)$th entry of $H$.
\end{corollary} 
\begin{proof}
The proof is the same as the proofs of Lemma~\ref{lem:gen_eig} and Theorem~\ref{lem:gen_eig} with $W$ replaced by $H$ throughout. 
\end{proof}

As in Section~\ref{subsec:gen_eig}, we see that although the block GMRES residual can be expressed in terms of a weighted polynomial approximation, there is no
reason for any of the weights to be small. However, tuning can also be applied within block GMRES to accelerate convergence, as in~\cite{ElmXue11}. 

\section{Numerical experiments} \label{sec:num_ex}
In this section we consider GMRES convergence within inverse iteration and the bounds previously 
discussed for nonsymmetric matrices whose properties are summarized in Table~\ref{t:ex1}.
The matrix \cdfd~is taken from~\cite{GolubYe00, frei07}, whereas \cdde{} and \olm{}
~are from the matrix market. 
In Table~\ref{t:ex1}, the spectral norms and condition numbers of the eigenvector matrices of the unpreconditioned $B$ as well as $BP^{-1}$ for different
standard preconditioners 
are listed. The standard preconditioners $P$ are incomplete LU factorizations with three different drop tolerances $\theta$.
To mimic a similar increasing quality of the polynomial preconditioners, the polynomial degrees are set to $d=5,10,15$. (Note that the polynomial
preconditioned matrices have the same eigenvector matrix as $B$.) 
For the coefficients of $p(B)$, we first generate a small number of Ritz and harmonic Ritz values of $B$.
If these Ritz values are all real or only have small imaginary parts, the coefficient generation \eqref{polycoeffcheb} based on the reciprocal Chebychev nodes
is used, whereas the approach \eqref{polycoeffmin} by \cite{Gij95} is used in the case of complex Ritz values. The standard and polynomial preconditioners are kept
unchanged during the outer iteration, i.e.\ they are computed only once at the start.
We employ two variants of the tuned preconditioners, which satisfy $P_iy_i=y_i$ and $P_iy_i=Ay_i$.
We also investigated tuning with $P_iy_i = (\gamma^{(i)} - \sigma) y_i$ as in~\eqref{eq:tunedpract} but the results were similar to the simpler $P_iy_i=y_i$ and
so have not been included. 

Note that we selected test examples of comparably small sizes in order to be able to compute the eigendecompositions needed for the weight
vectors $w$. The effects and performance gains resulting from the application of tuned preconditioners have been demonstrated with large matrices, e.g.\ in
\cite{SzyX11,Mar16}.

For each problem we run inverse iteration with an 
initial eigenvector approximation of $y_0 = \frac{1}{n}[ 1, \ 1, \ \dotsc, \ 1]^T$. 
We let $\rho_i$ be the inverse iteration residual, so that 
\begin{equation}\label{eq:iires}
\rho_i = Ax_{i} - \lambda_{i} x_{i}.
\end{equation}
The shifts $\sigma$ are as in Chapter 6 of~\cite{frei07}.
We solve the linear system using GMRES with a zero initial guess and a decreasing tolerance of 
$\tau = \min\{\delta,\delta\|\rho_{i-1}\|_2\}$. 
We choose $\delta = 0.1$ except for \cdde{} ($\delta = 0.001$); 
these are the same values as in Chapter~6 of~\cite{frei07}.

\begin{table}
\footnotesize
\setlength{\tabcolsep}{0.5em}
\centering
\caption{Matrices used in examples.}
\label{t:ex1}
\begin{tabular}{l r r r |r r | rr | rr }
& & \multicolumn{2}{c}{Unpreconditioned} & \multicolumn{2}{|c}{$P$ ($10^{-1}$)} & \multicolumn{2}{|c}{$P$ ($10^{-2}$)} & \multicolumn{2}{|c}{$P$ ($10^{-3}$)}\\
Matrix & $n$ & $\|Z\|_2$ & $\kappa_2(Z)$ & $\|Z\|_2$ & $\kappa_2(Z)$ & $\|Z\|_2$ & $\kappa_2(Z)$ & $\|Z\|_2$ & $\kappa_2(Z)$\\
\hline
\input{Tab/IICond.tex}
\end{tabular}
\end{table}
At first we have a detailed look at the progress of the outer iteration for the matrix \cdde{}~using no preconditioner,
ILU with $\theta=10^{-2}$ and the corresponding two tuned variants, and a polynomial preconditioner with deg$(p)=d=10$. 
 Tables~\ref{t:cdde:U}--\ref{t:cdde:T} summarize the changes in the relevant quantities as the outer iteration proceeds, including the quality of the
eigenpair approximations, the components in $w^{[i]}$, the constants in the GMRES bounds \eqref{eq:2.5} and~\eqref{eq:2.6}, and the number of executed GMRES steps. For the tuned preconditioners, also the progress of the spectral norms and condition numbers of the eigenvector matrices is given. Note that for clarity we drop the outer iteration index $i$ from $w$ for the remainder of this section.

\begin{table}
\footnotesize
\setlength{\tabcolsep}{0.5em}
\centering
\caption{\texttt{cdde1} without preconditioning, with untuned ILU preconditioning (drop tolerance $10^{-2}$) and with polynomial preconditioning 
(polynomial degree 10): inverse iteration residual $\|\rho_{i-1}\|_2$, sizes of components of $w$, constants in~\eqref{eq:2.5} and~\eqref{eq:2.6} and 
GMRES iteration numbers at each outer iteration $i$.}
\label{t:cdde:U}
\begin{tabular}{*{10}{l}}
 & $i$ & $\lambda^{(i)}$ & $\|\rho_{i-1}\|_2$ & $|w_1|$ & $\|{\wt}\|_2$ & $\|Z\|_2\|w\|_2$ &  $\|Z\|_2 \|\tilde w\|_2$ & $\tfrac{\|Z\|_2\|\tilde w\|_2}{\tau}$
&it\\
\hline 
\input{./Tab/cdde11e-2Comb1.tex}
\end{tabular}
\end{table}

\begin{table}
\scriptsize
\setlength{\tabcolsep}{0.25em}
\centering
\caption{\texttt{cdde1} with tuned preconditioning ($Py_i = Iy_i$ and $Py_i = Ay_i$): inverse iteration residual $\|\rho_{i-1}\|_2$, eigenvector matrix norm
$\|Z\|_2$ and condition number $\kappa_2(Z)$, sizes of components of $w$, constants in~\eqref{eq:2.5} and~\eqref{eq:2.6} and 
GMRES iteration numbers at each outer iteration $i$. The ILU drop tolerance is $10^{-2}$.}
\label{t:cdde:T}
\begin{tabular}{*{12}{l}}
 &$i$&$\lambda^{(i)}$&$\|\rho_{i-1}\|_2$&$\|Z\|_2$&$\kappa_2(Z)$&$|w_1|$&$\|{\wt}\|_2$&
 {\smaller $\|Z\|_2\|w\|_2$}
 &
 {\smaller $\|Z\|_2\|\tilde w\|_2$}
 &$\tfrac{\|Z\|_2\|\tilde w\|_2}{\tau}$&it\\
\hline 
\input{./Tab/cdde11e-2Comb2.tex}
\end{tabular}
\end{table}

\begin{figure}
\centering
\begin{subfigure}[b]{0.45\textwidth}
\includegraphics[width=\textwidth,trim = 1cm 6.5cm 1cm 6.5cm,clip=true]{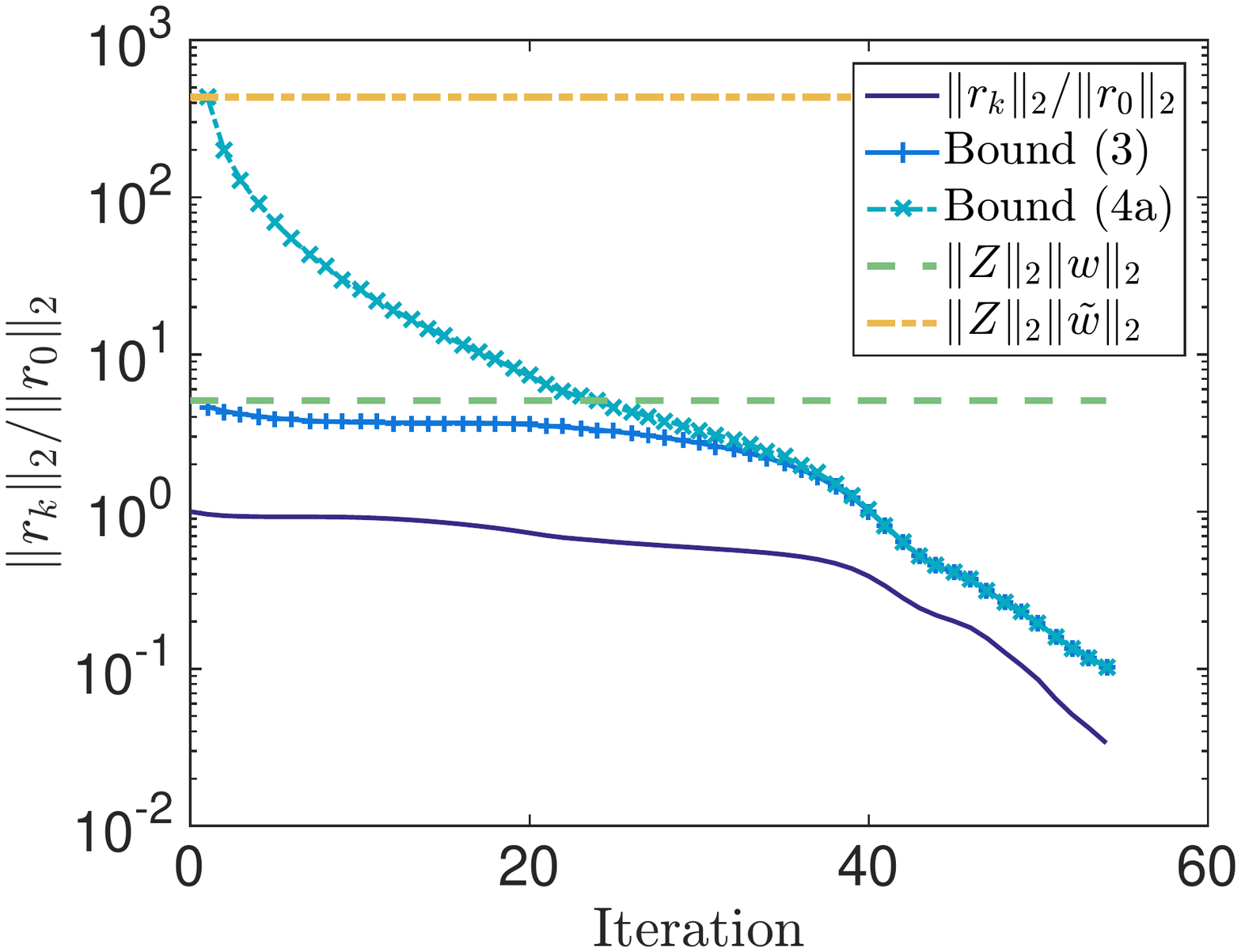}
\caption{No preconditioner, $k=1$}
\end{subfigure}
\begin{subfigure}[b]{0.45\textwidth}
\includegraphics[width=\textwidth,trim = 1cm 6.5cm 1cm 6.5cm,clip=true]{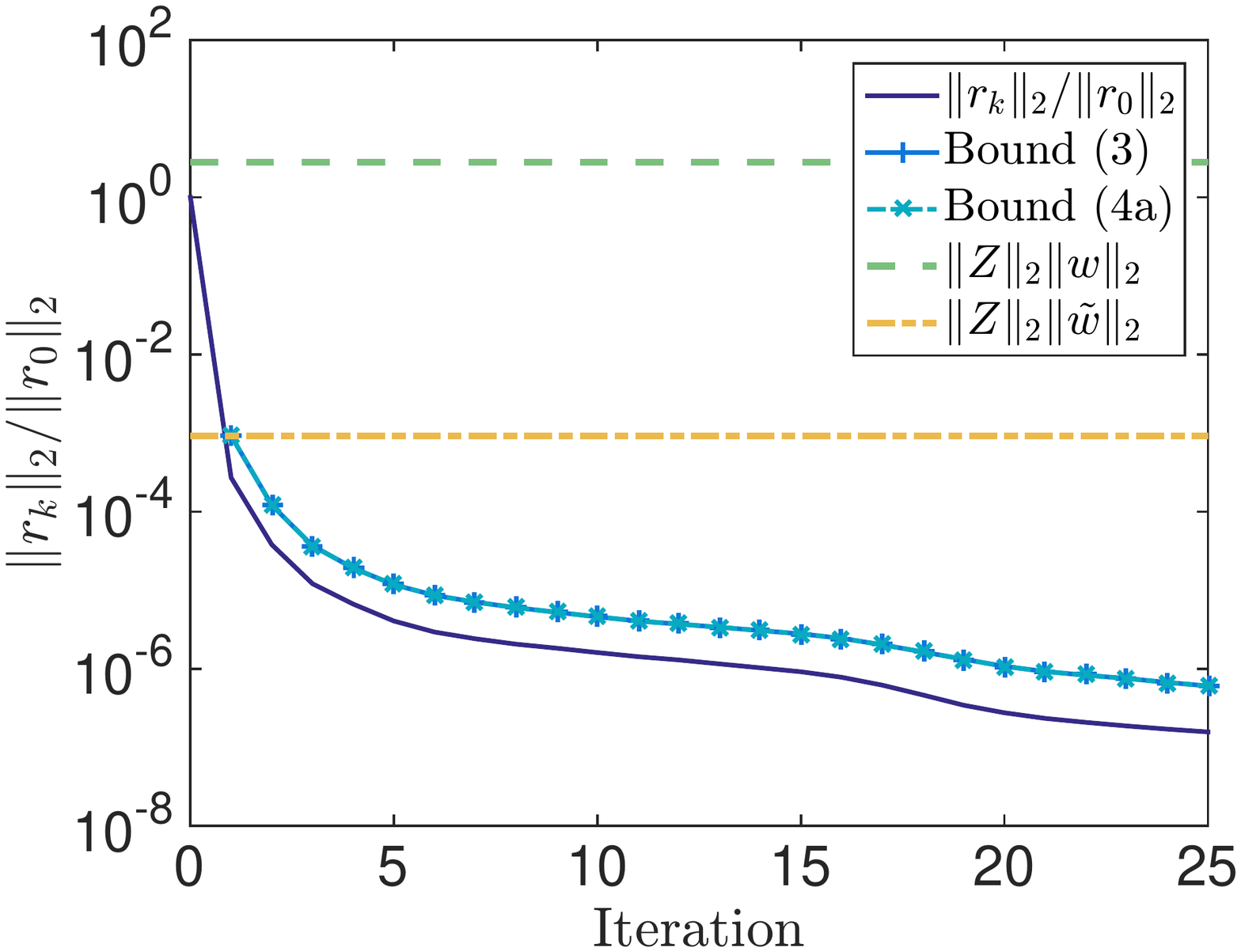}
\caption{No preconditioner, $k=8$}
\end{subfigure}

\begin{subfigure}[b]{0.45\textwidth}
\includegraphics[width=\textwidth,trim = 1cm 6.5cm 1cm 6.5cm,clip=true]{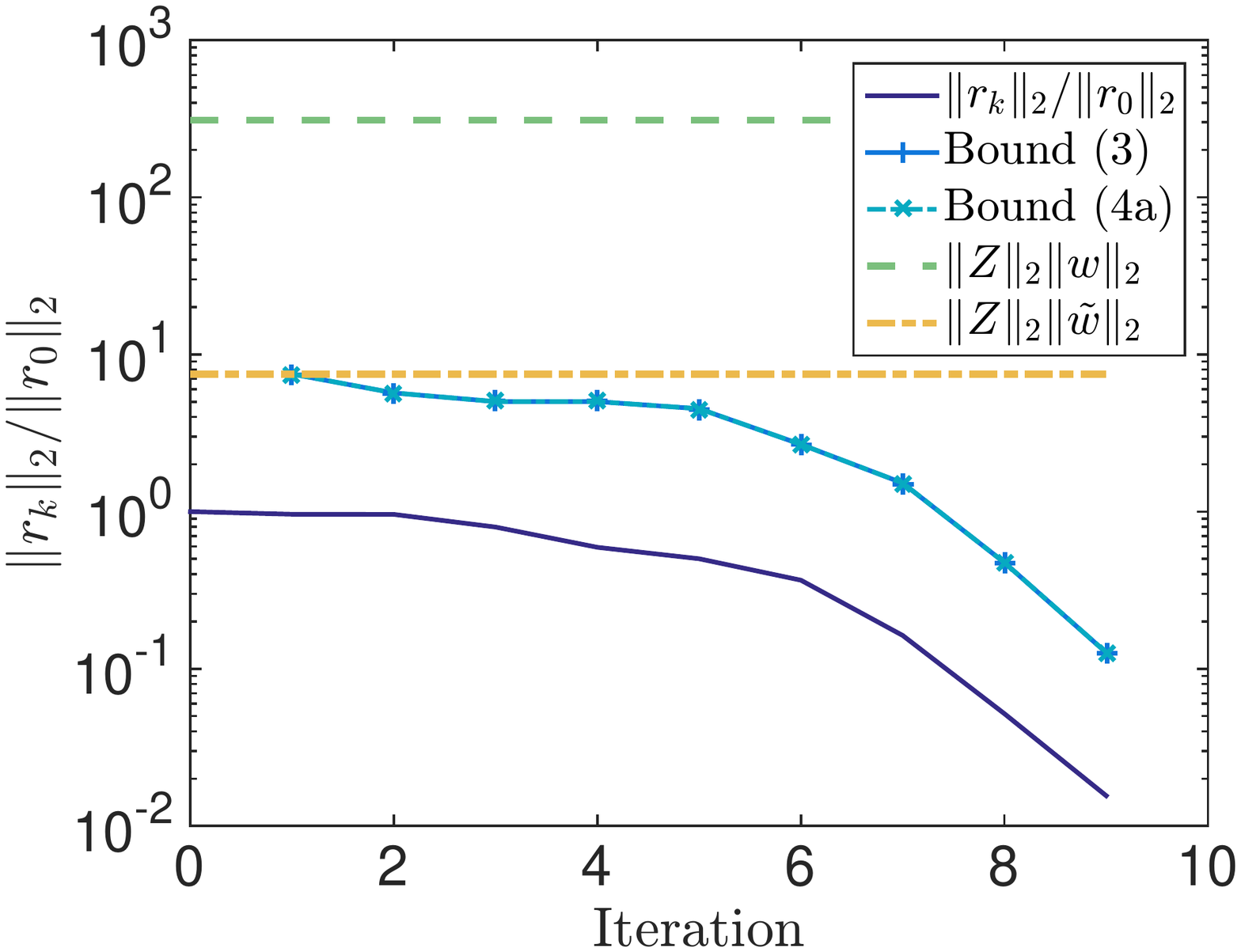}
\caption{Tuned ILU ($Py_i = y_i$), $k=1$}
\end{subfigure}
\begin{subfigure}[b]{0.45\textwidth}
\includegraphics[width=\textwidth,trim = 1cm 6.5cm 1cm 6.5cm,clip=true]{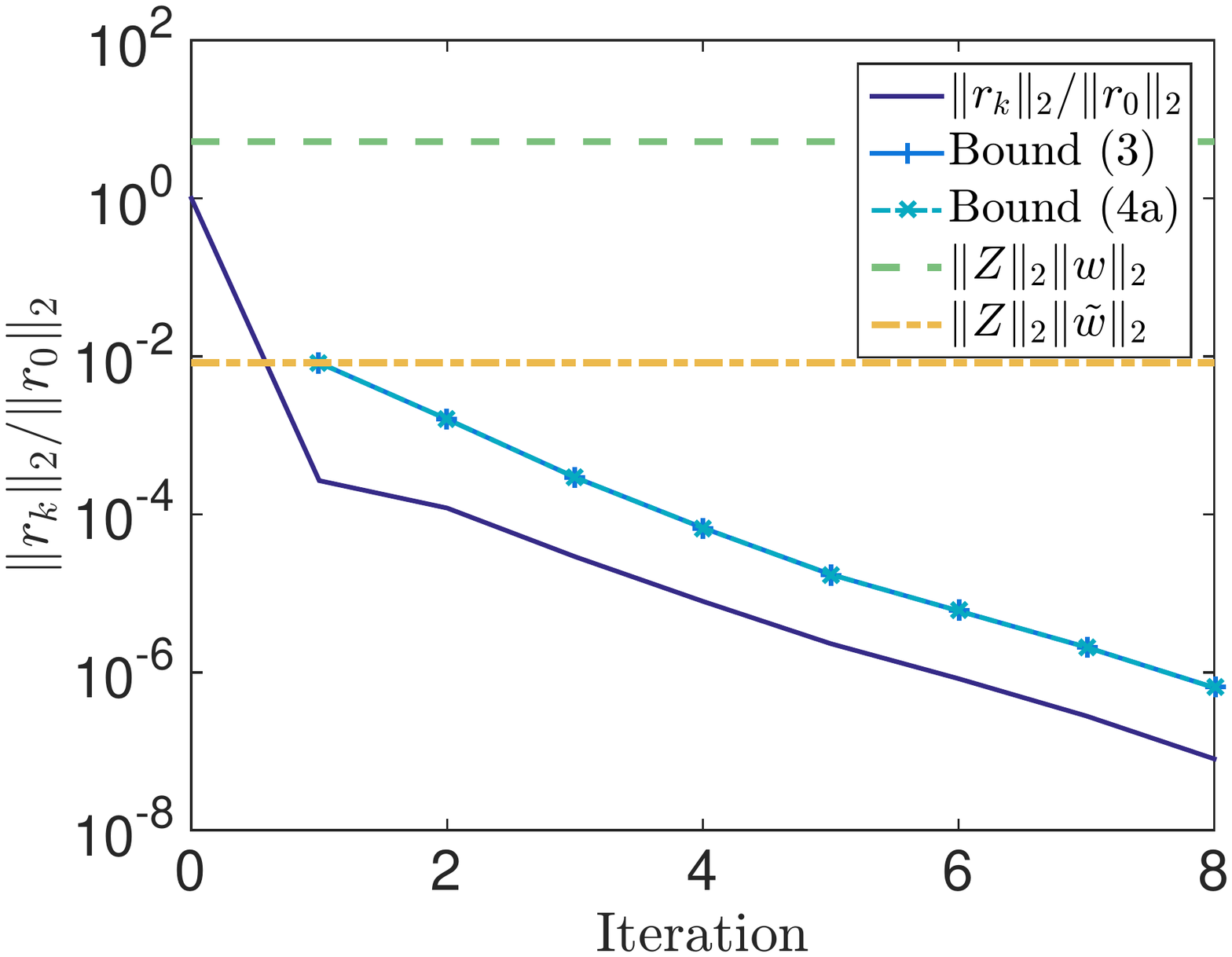}
\caption{Tuned ILU ($Py_i = y_i$), $k=8$}
\end{subfigure}

\begin{subfigure}[b]{0.45\textwidth}
\includegraphics[width=\textwidth,trim = 1cm 6.5cm 1cm 6.5cm,clip=true]{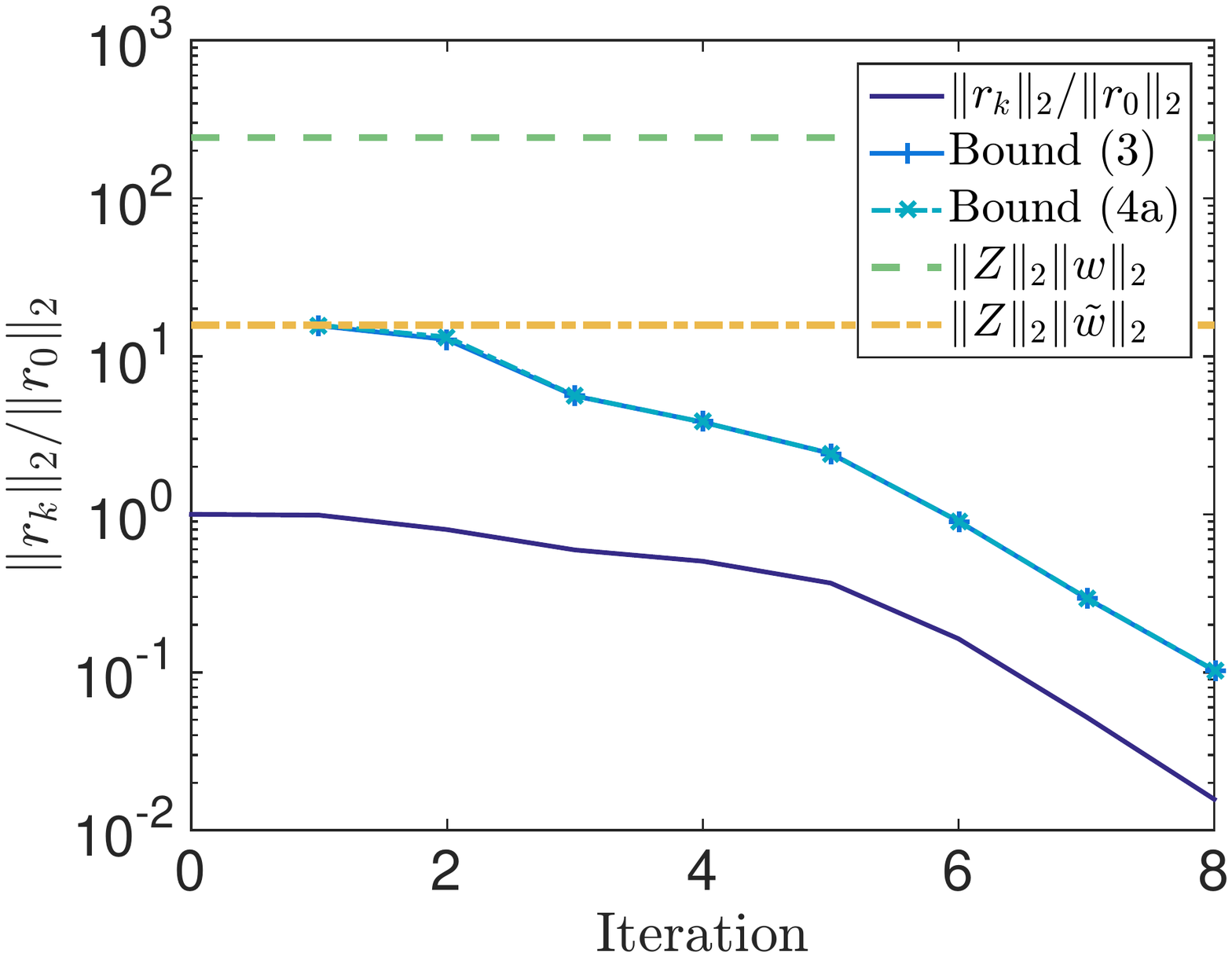}
\caption{Tuned ILU ($Py_i = Ay_i$), $k=1$}
\end{subfigure}
\begin{subfigure}[b]{0.45\textwidth}
\includegraphics[width=\textwidth,trim = 1cm 6.5cm 1cm 6.5cm,clip=true]{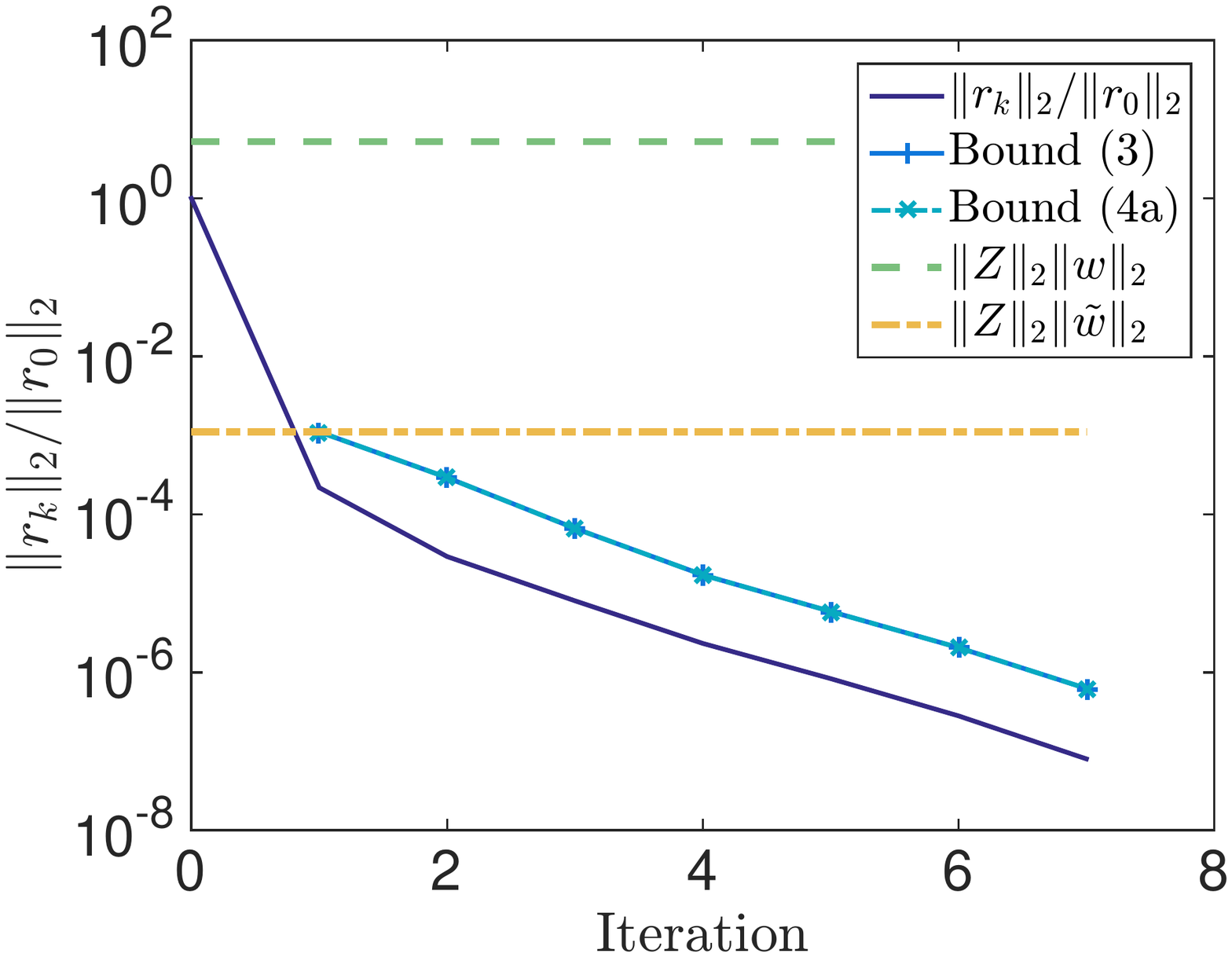}
\caption{Tuned ILU ($Py_i = Ay_i$), $k=8$}
\end{subfigure}

\begin{subfigure}[b]{0.45\textwidth}
\includegraphics[width=\textwidth,trim = 1cm 6.5cm 1cm 6.5cm,clip=true]{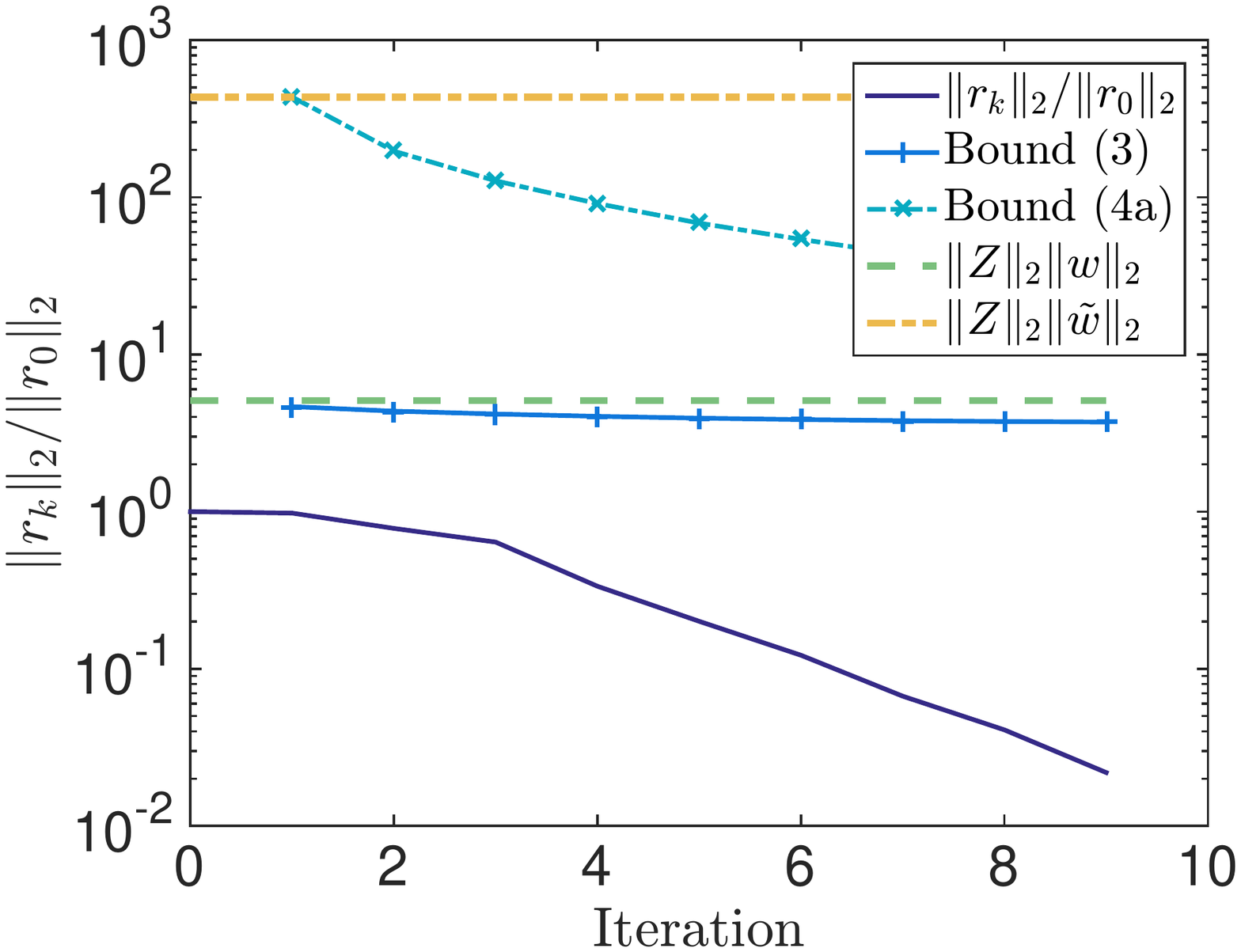}
\caption{Polynomial, $k=1$}
\end{subfigure}
\begin{subfigure}[b]{0.45\textwidth}
\includegraphics[width=\textwidth,trim = 1cm 6.5cm 1cm 6.5cm,clip=true]{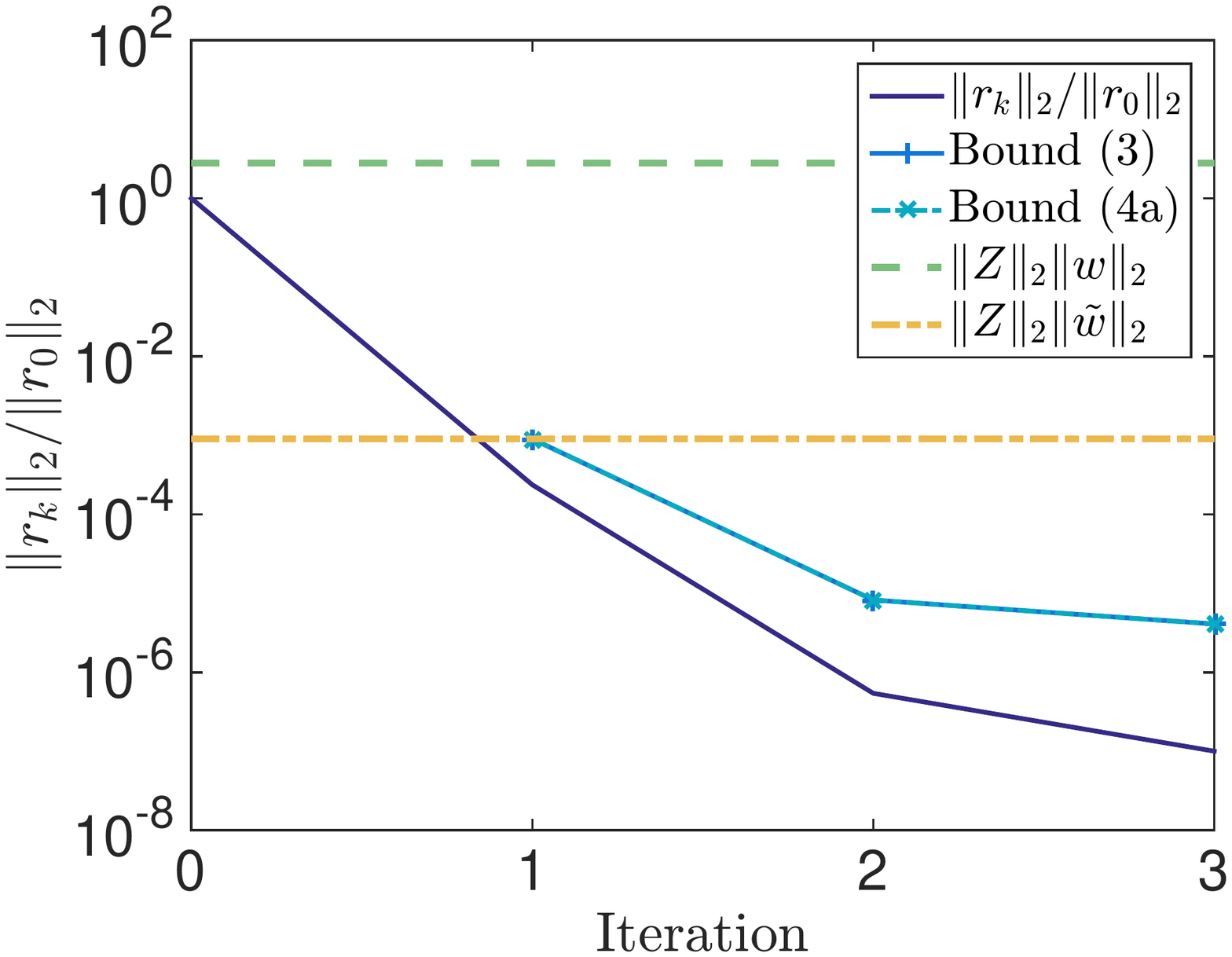}
\caption{Polynomial, $k=8$}
\end{subfigure}

\caption{Relative residual norms, bounds~\eqref{eq:2.5} and \eqref{eq:2.6a}, and the constants $\|Z\|_2\|w\|_2$ and $\|Z\|_2\|\tilde w\|_2$ for \texttt{cdde1}
for different preconditioners. The ILU drop tolerance is $10^{-2}$ and deg$(p)=10$.}
\label{f:cdde1_residuals}
\end{figure}

Obviously, the values of $\lambda^{(i)}$, $\|\rho_{i}\|_2$ at each outer iteration reveal that using different  preconditioners does not hamper  the convergence
of
the eigenpairs. It is also apparent that $w\rightarrow e_1$, as the outer iteration $i$ proceeds, in the unpreconditioned case (Table~\ref{t:cdde:U}). This
behavior is somehow destroyed
by the 
standard preconditioner (Table~\ref{t:cdde:U}) but mimicked by the two tuned preconditioners (Table~\ref{t:cdde:T}) and exactly
matched for the polynomial one (Table~\ref{t:cdde:U}). Using either tuned or polynomial preconditioners also has a positive effect on the constants in 
\eqref{eq:2.5} and~\eqref{eq:2.6}, which are larger for the standard preconditioner. Most importantly, the number of GMRES iteration steps is notably
reduced and remains at an approximately constant level for tuned and polynomial preconditioners. The increasing trend for the standard preconditioner is also
evident. 

A visual illustration of these observations is given in Figures~\ref{f:cdde1_residuals} and \ref{f:cdde:1e-2}. Figure~\ref{f:cdde1_residuals} shows the
behavior of the GMRES residuals for no preconditioner,  standard, tuned, and polynomial preconditioners. As discussed in Section~\ref{sec:initdecrease} we
observe the initial decrease of both the GMRES residual and the bound in \eqref{eq:2.6a}  (when no preconditioner or the tuned preconditioner is used), which is
particularly prominent in the later stages of the iteration.
Figure~\ref{f:cdde:1e-2} shows the history of components of $w$ as the outer iteration proceeds in
the two top plots.
The bottom plots show the required number of GMRES steps against the outer iteration (left plot) as well as the eigenvalue residual norm against the
cumulative sum of inner GMRES steps (right plot). The significant reduction of inner iterations by tuned and polynomial preconditioners is apparent.

\begin{figure}
    \centering
    \begin{subfigure}[b]{0.45\textwidth}
        \includegraphics[width=\textwidth,trim = 1cm 6.5cm 1cm 6.5cm,clip=true]{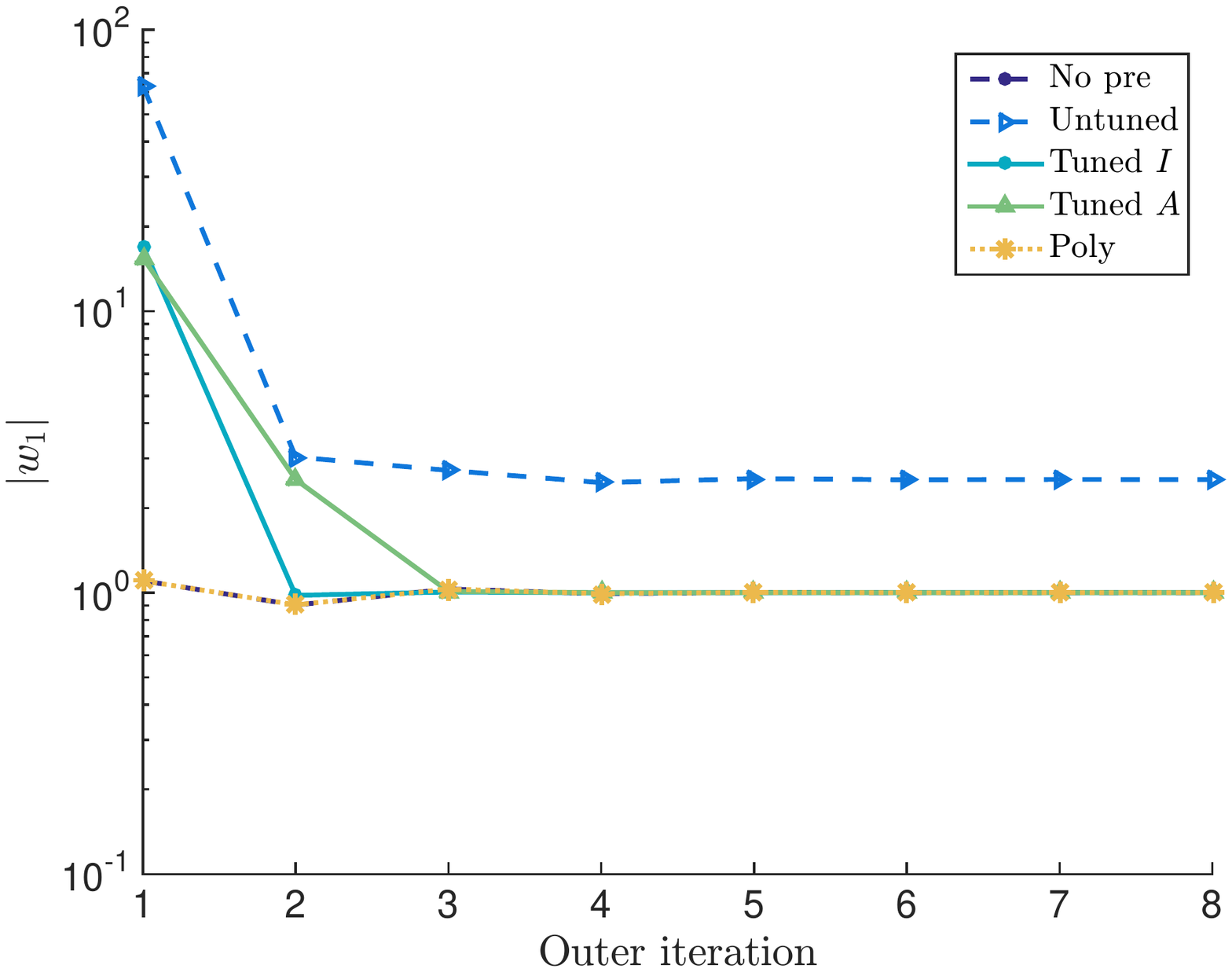}
    \end{subfigure}
    \begin{subfigure}[b]{0.45\textwidth}
        \includegraphics[width=\textwidth,trim = 1cm 6.5cm 1cm 6.5cm,clip=true]{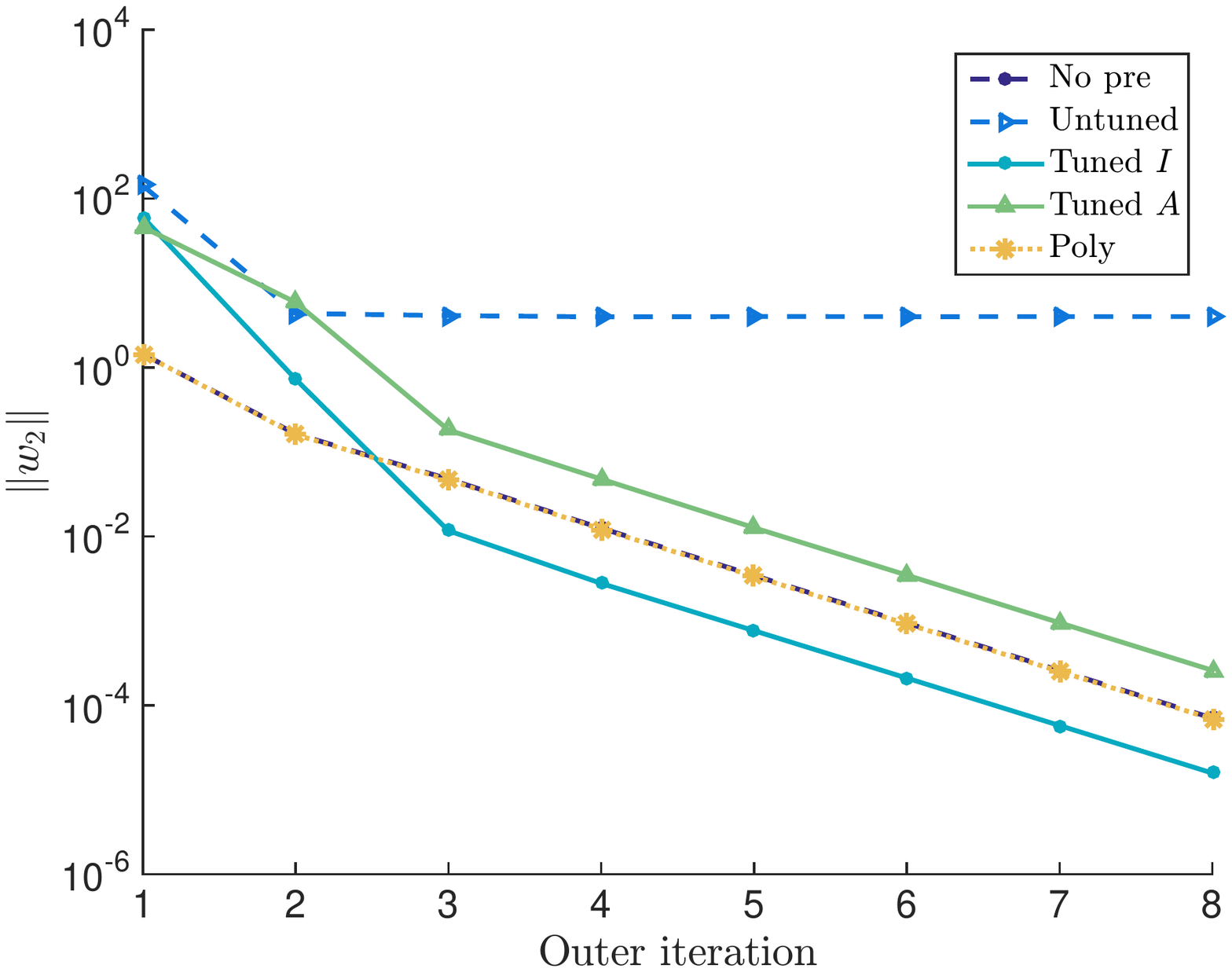}
    \end{subfigure}
    \begin{subfigure}[b]{0.45\textwidth}
        \includegraphics[width=\textwidth,trim = 1cm 6.5cm 1cm 6.5cm,clip=true]{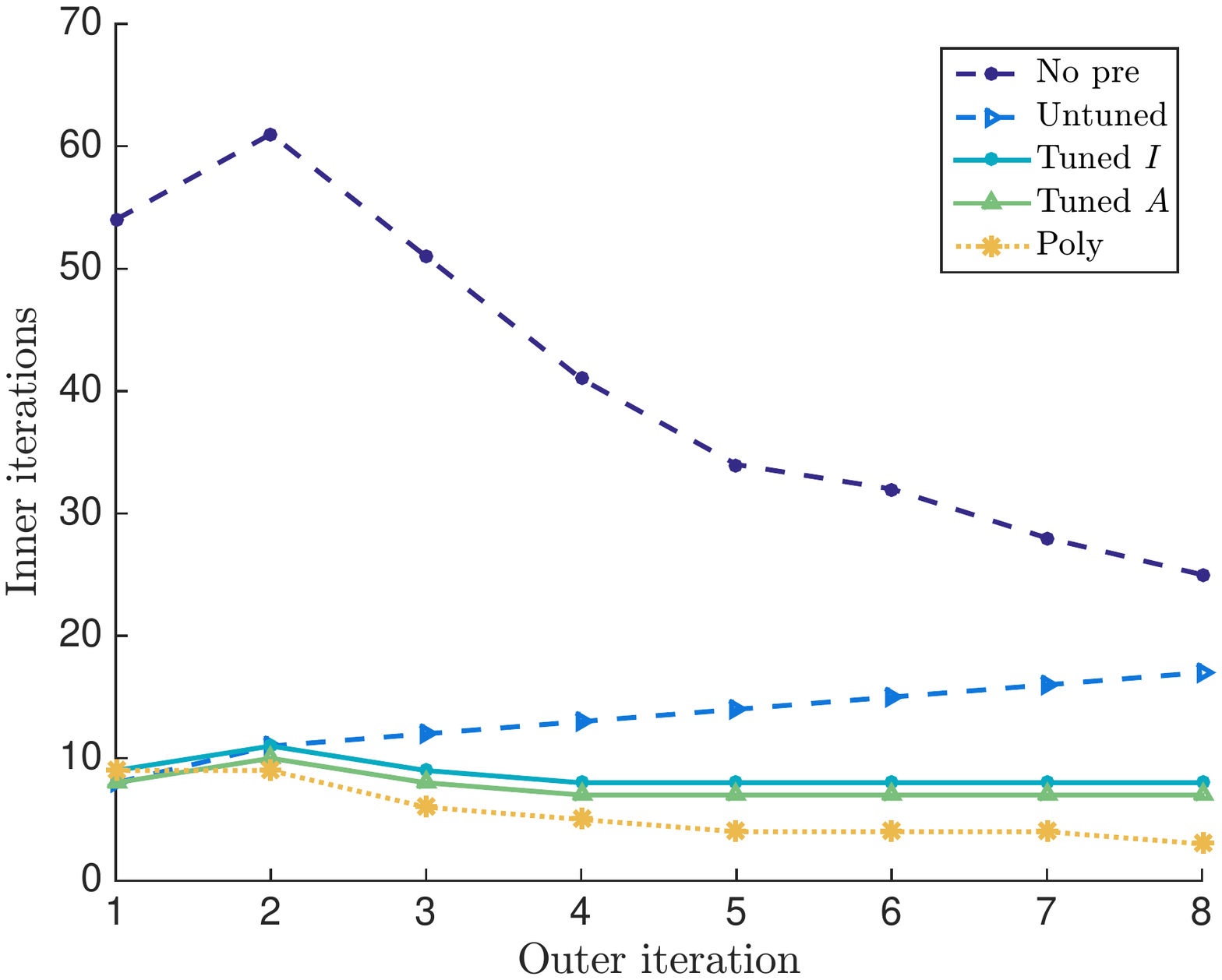}
    \end{subfigure}
        \begin{subfigure}[b]{0.45\textwidth}
        \includegraphics[width=\textwidth,trim = 1cm 6.5cm 1cm 6.5cm,clip=true]{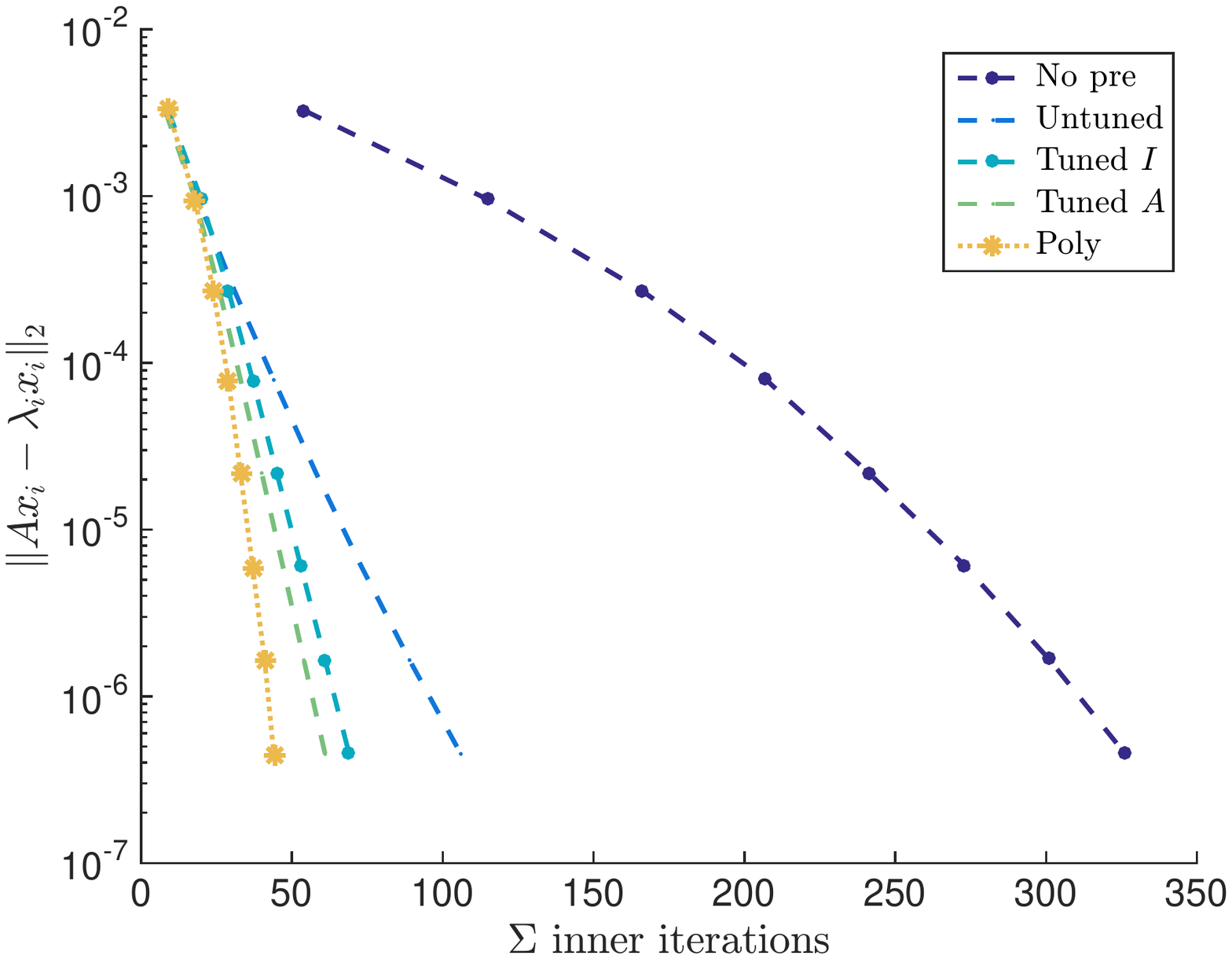}
    \end{subfigure}
    \caption{\texttt{cdde1}: $|w_1|$, $\|w_2\|$, number of GMRES iterations and outer residual. The ILU drop tolerance is $10^{-2}$ and deg$(p)=10$.}
    \label{f:cdde:1e-2}
\end{figure}

Figures~\ref{f:cdfd:1e-2} and \ref{f:olm:1e-2} show similar plots for the matrices \cdfd{}~and \olm{} using the same preconditioning settings.
For \cdfd{}, using a standard preconditioner results in a drastic increase in the magnitude of the components of the weight vector $w$ (top plots of
Figure~\ref{f:cdfd:1e-2}). The beneficial effects of tuned and polynomial preconditioners are similar to the previous examples.
For the matrix \olm{}, the tuned preconditioners lead to a slight increase of the outer iteration steps (14 compared to 11 for the other choices).
However, the amount of work in terms of the number of required inner iteration steps is still smaller than for the other variants (bottom plots of
Figure~\ref{f:olm:1e-2}). The polynomial preconditioner seems to be of lesser quality compared to other preconditioners for the matrix \olm{}, leading to more
inner iterations. Increasing the polynomial degree did not lead to improvements. It seems that for this example, the basic strategies mentioned in
Section~\ref{ssec:polyprec} to select the coefficients of the preconditioning polynomial are not sufficient. For these cases, this highlights an advantage of
tuned preconditioners over polynomial preconditioning, especially for tuned preconditioners that are built from standard preconditioning approaches, e.g.
incomplete factorizations, which can be constructed in a much more automatic and
straightforward manner. 

\begin{figure}
    \centering
    \begin{subfigure}[b]{0.45\textwidth}
        \includegraphics[width=\textwidth,trim = 1cm 6.5cm 1cm 6.5cm,clip=true]{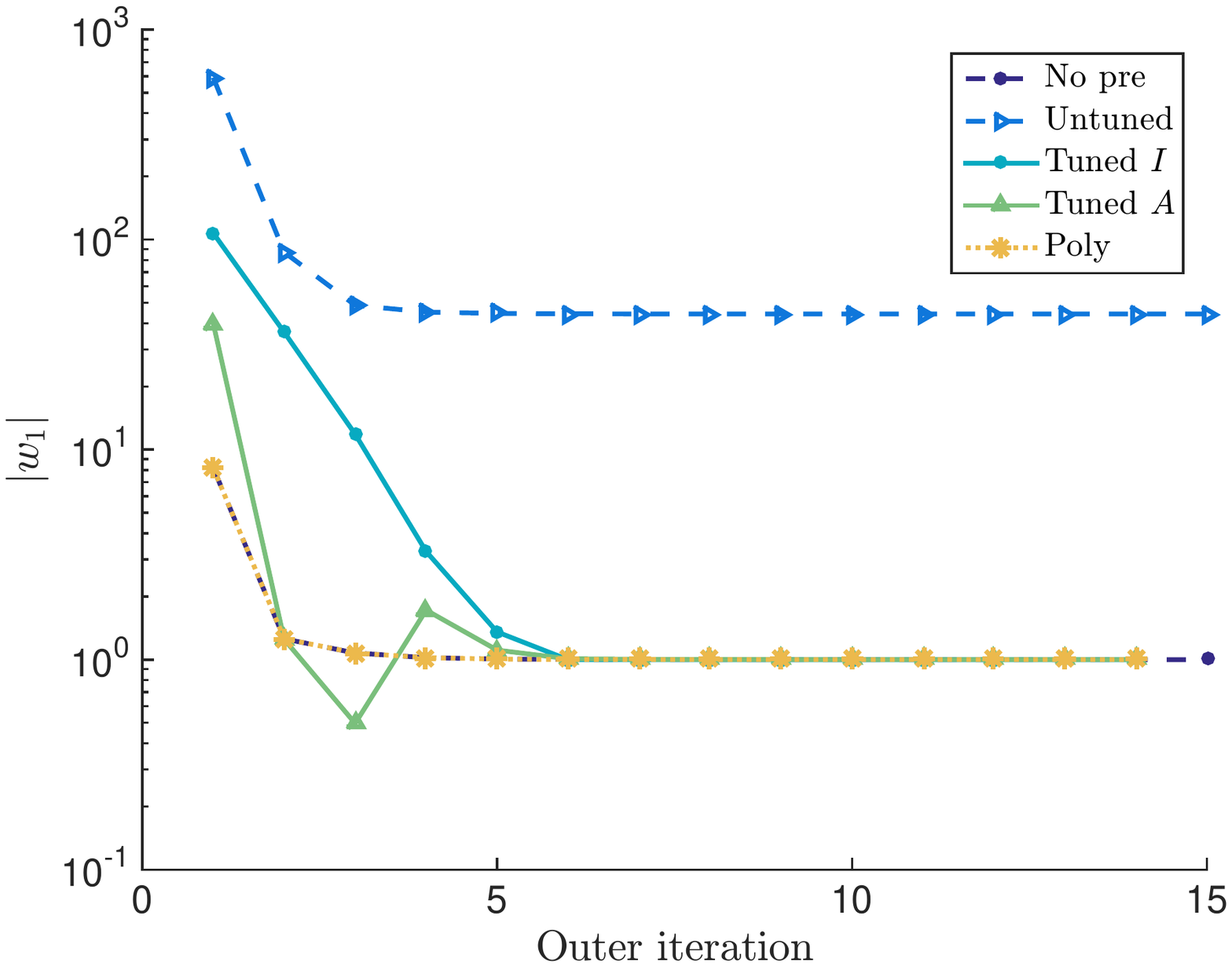}
    \end{subfigure}
    \begin{subfigure}[b]{0.45\textwidth}
        \includegraphics[width=\textwidth,trim = 1cm 6.5cm 1cm 6.5cm,clip=true]{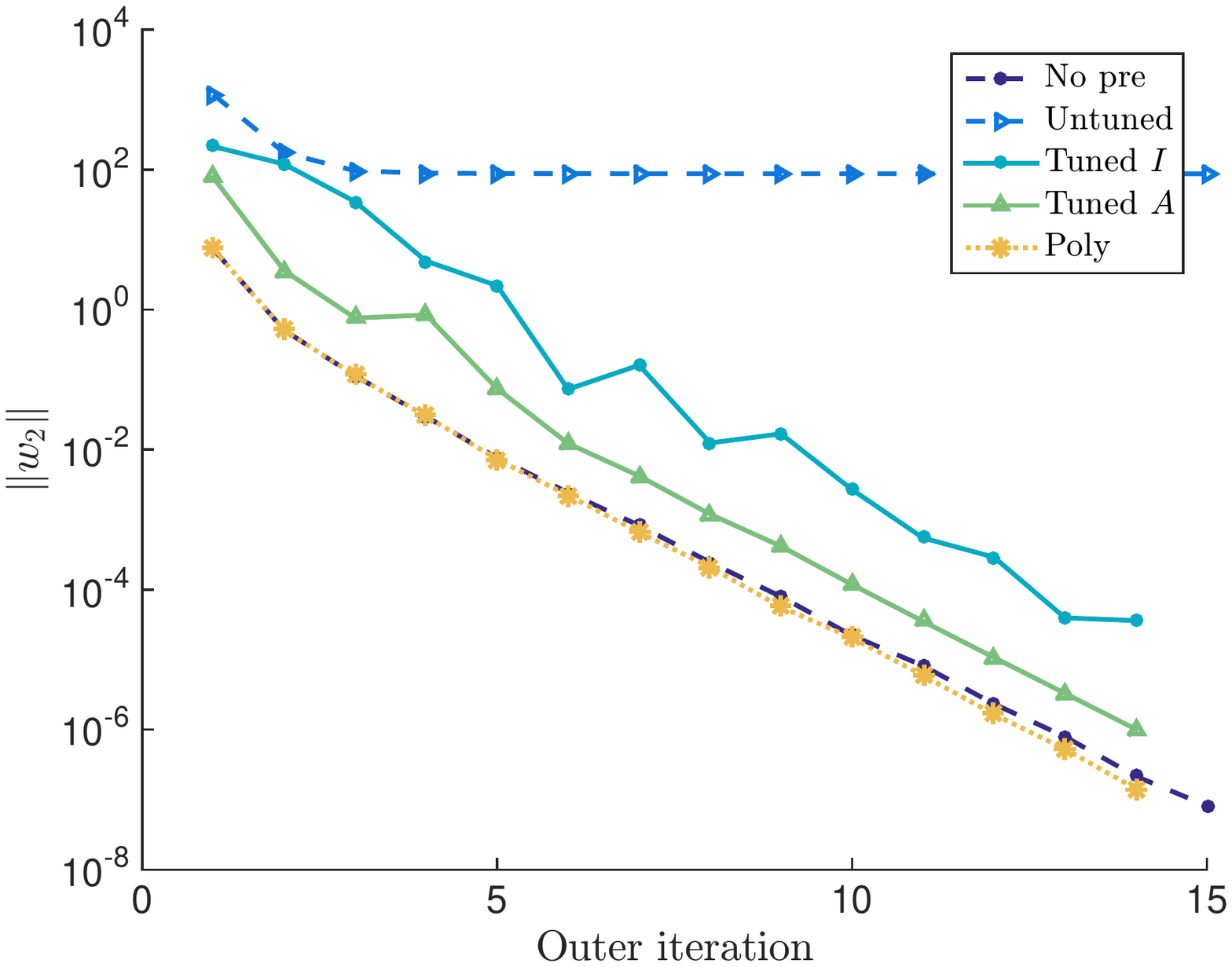}
    \end{subfigure}
    \begin{subfigure}[b]{0.45\textwidth}
        \includegraphics[width=\textwidth,trim = 1cm 6.5cm 1cm 6.5cm,clip=true]{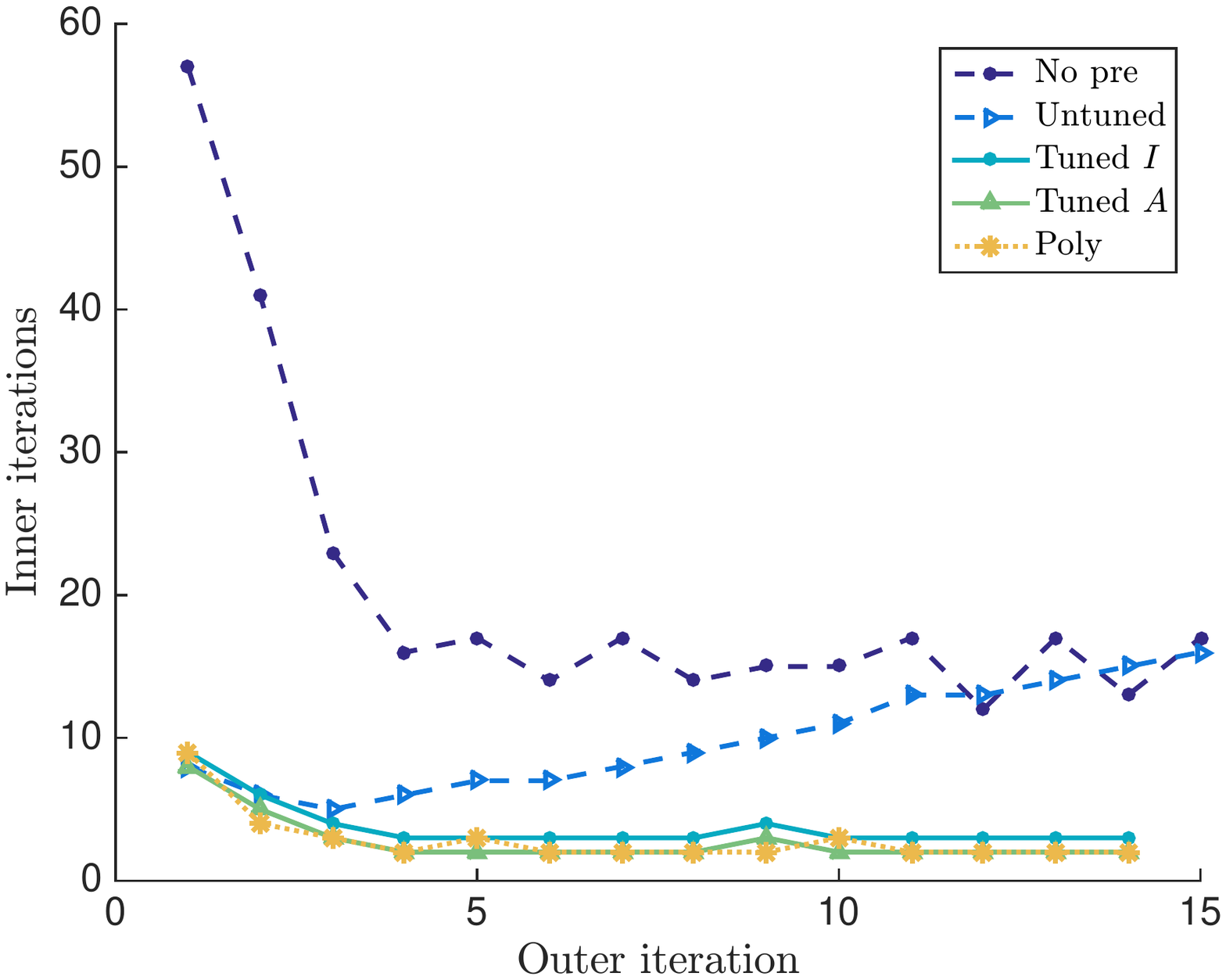}
    \end{subfigure}
        \begin{subfigure}[b]{0.45\textwidth}
        \includegraphics[width=\textwidth,trim = 1cm 6.5cm 1cm 6.5cm,clip=true]{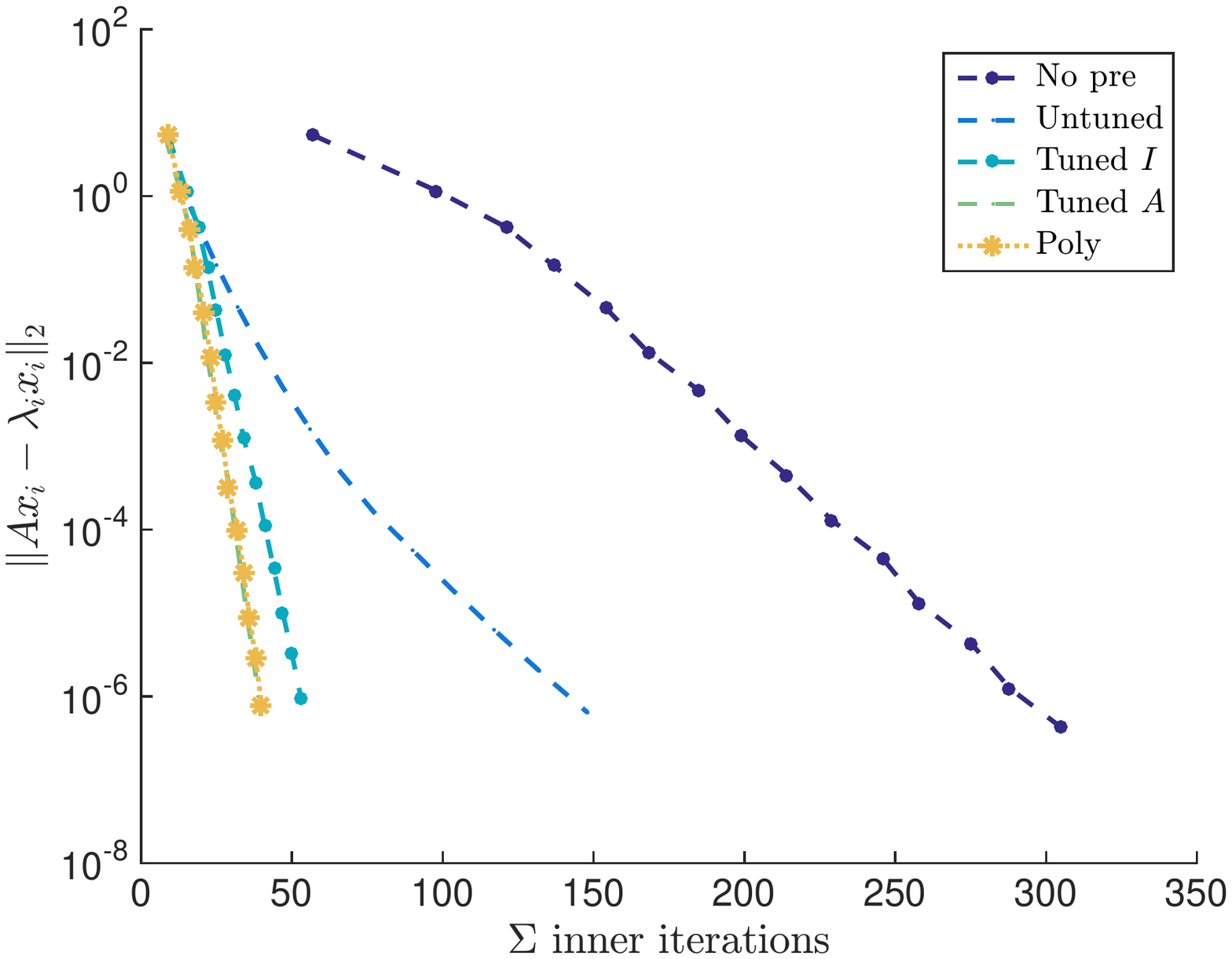}
    \end{subfigure}
    \caption{\texttt{cd\_fd}: $|w_1|$, $\|w_2\|$, number of GMRES iterations and eigenvalue residual. The ILU drop tolerance is $10^{-2}$ and deg$(p)=10$.}
     \label{f:cdfd:1e-2}
\end{figure}

\begin{figure}
    \centering
    \begin{subfigure}[b]{0.45\textwidth}
        \includegraphics[width=\textwidth,trim = 1cm 6.5cm 1cm 6.5cm,clip=true]{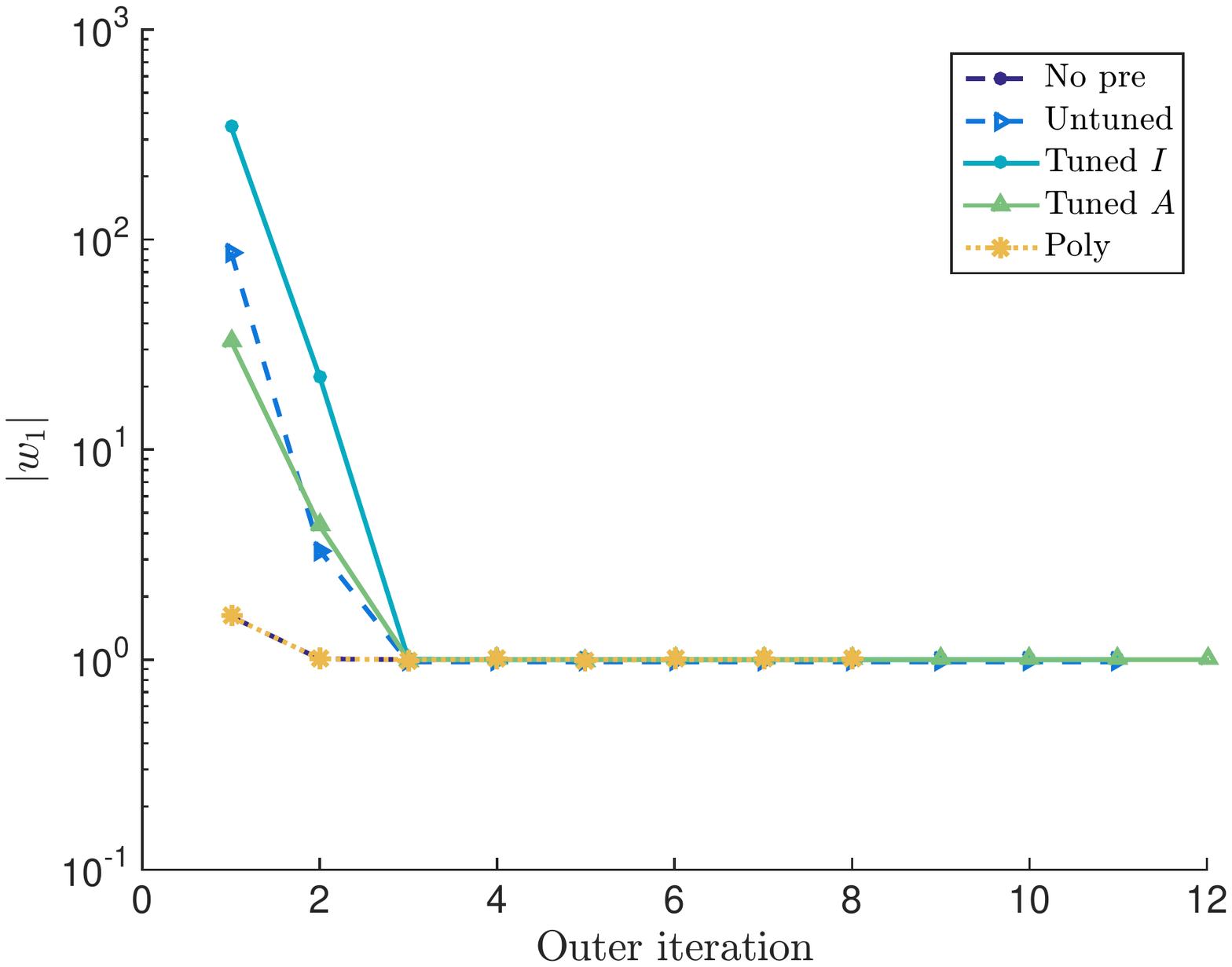}
    \end{subfigure}
    \begin{subfigure}[b]{0.45\textwidth}
        \includegraphics[width=\textwidth,trim = 1cm 6.5cm 1cm 6.5cm,clip=true]{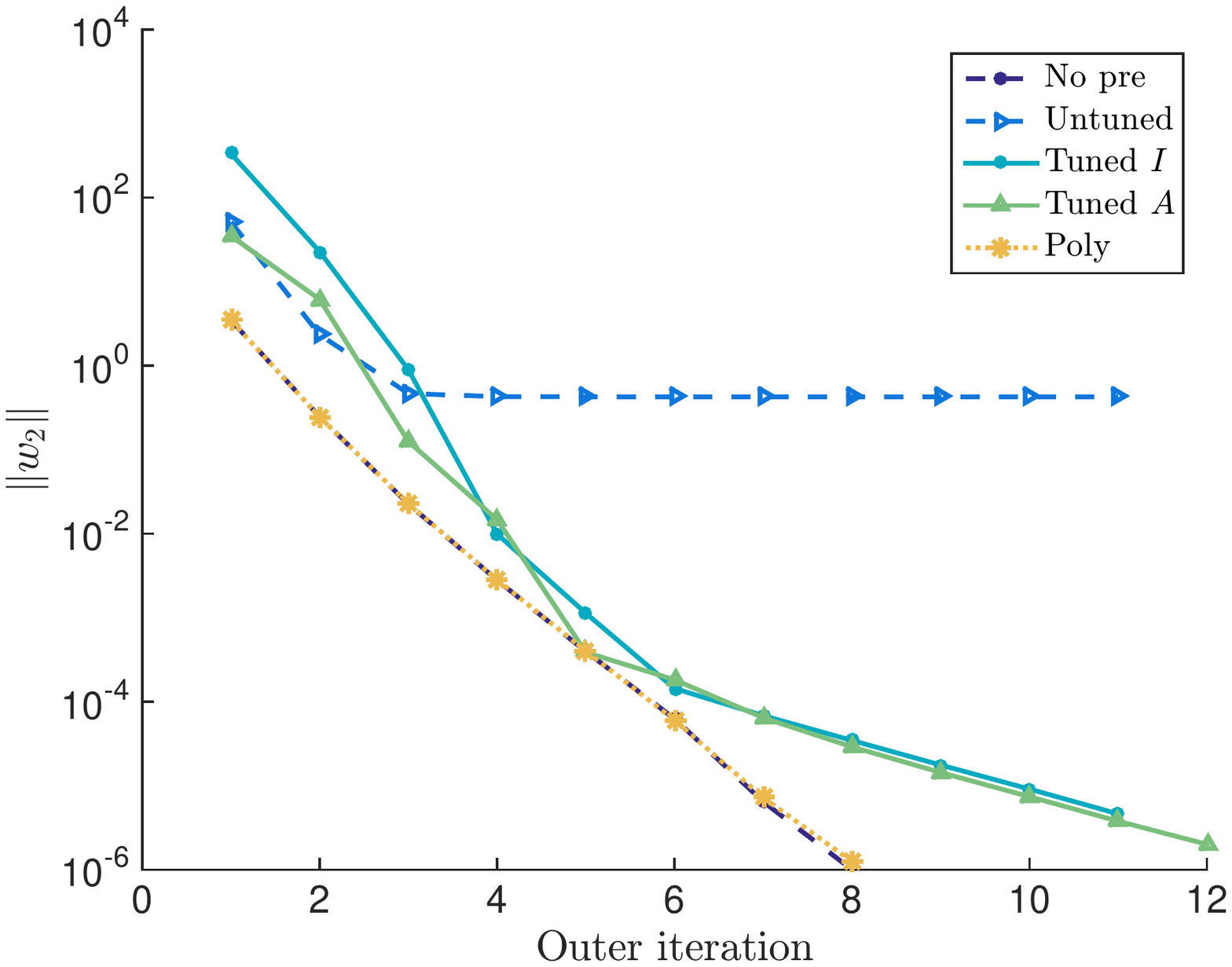}
    \end{subfigure}
    \begin{subfigure}[b]{0.45\textwidth}
        \includegraphics[width=\textwidth,trim = 1cm 6.5cm 1cm 6.5cm,clip=true]{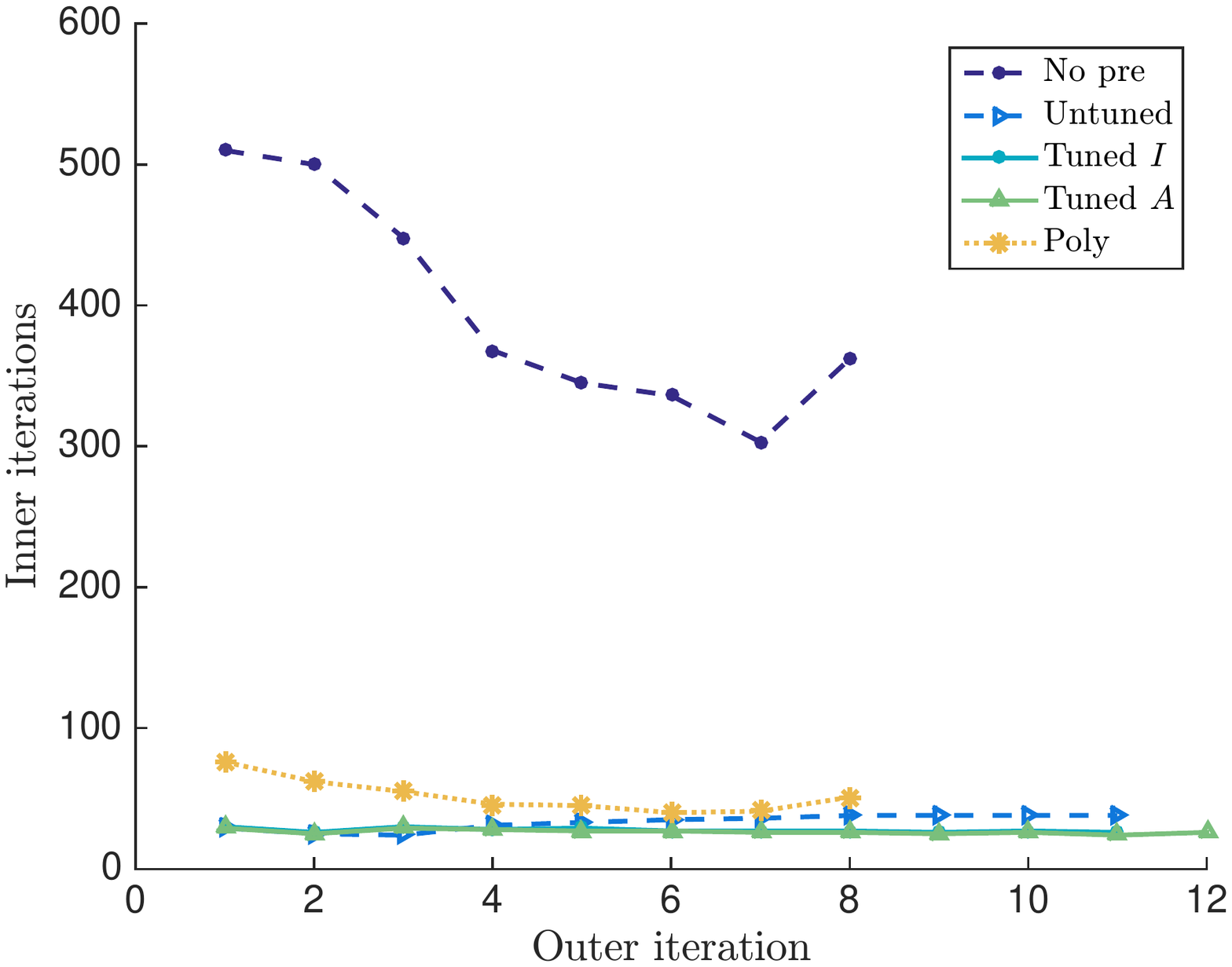}
    \end{subfigure}
    \begin{subfigure}[b]{0.45\textwidth}
            \includegraphics[width=\textwidth,trim = 1cm 6.5cm 1cm 6.5cm,clip=true]{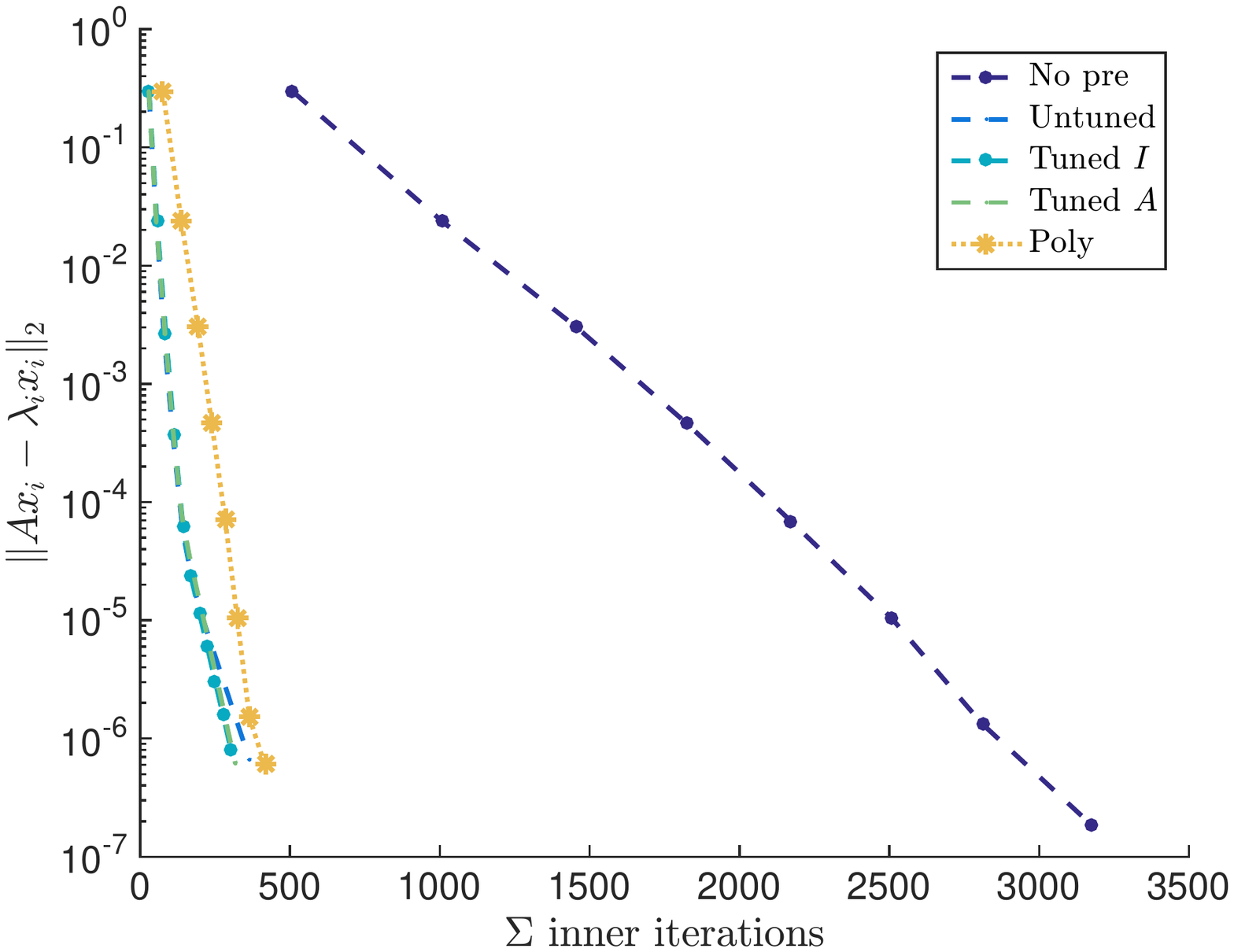}
    \end{subfigure}
    \caption{\texttt{olm2000}: $|w_1|$, $\|w_2\|$, number of GMRES iterations and eigenvalue residual. The ILU drop tolerance is $10^{-2}$ and deg$(p)=10$.}
     \label{f:olm:1e-2}
\end{figure}

\begin{table}[p]
\footnotesize
\centering
\caption{Results for all examples: final weight components $|w_1|$, $\|\wt\|_2$, and total numbers of outer and inner (GMRES)
iterations.}
\label{t:all}
\begin{tabular}{*{7}{l}}
ex.& prec. & $\theta, d$& final $|w_1|$ &final  $\|\wt\|_2$ & outer & $\sum$ inner\\
\hline 
\input{./Tab/results_all.tex}
\end{tabular}
\end{table}

For all examples and all used preconditioners, Table~\ref{t:all} gives a condensed overview of the obtained results, listing
the final magnitudes of the weight components in $w$, and the total numbers of outer and inner (GMRES) iterations.
While the majority of the results are in line with the previous observations, there are some outliers.
For the matrices \cdfd~and \cdde, the tuned preconditioners built from the ILU with drop tolerance $\theta=10^{-1}$ lead to much smaller magnitudes of the
weights
compared to standard preconditioner, but the components of $w$ do not have the property $|w_1|>\|w^{(2)}\|_2$. 
Decreasing the drop tolerance to $\theta=10^{-2},~10^{-3}$ leads to the desired situation $w\rightarrow e_1$. 
Moreover, increasing the degree of the polynomial preconditioner is not always beneficial as we see in all three examples for $d=15$.
With this setting the polynomial preconditioner leads to a worse performance compared to lower polynomial degrees.
As we mentioned above, the origin of this behavior can is the basic coefficient generation approaches, which appear to return more and more very tiny
coefficients if the degree is increased. Given the difficulties with the polynomial preconditioner in these cases, we recommend to either 
look for better coefficient selection strategies or use a tuned
preconditioner.

We now transform \cdfd{} into a generalized problem by adding an artificial, tridiagonal matrix $M$ having the values 2.5, 5, 2.5 on the lower, main, and upper
diagonal, respectively. The shift is $\sigma=30$ but the remaining settings to execute this test are  unchanged. The result are illustrated in
Figure~\ref{f:Gcdfd}. As predicted, the two upper plots show that, in contrast to the situation $M=I$, the weight vector 
in the unpreconditioned and polynomial preconditioned case does not converge to $e_1$. This weight behavior is only achieved by employing tuned preconditioners.
The bottom left plot also shows a increasing number of inner iterations when no or a polynomial preconditioner is used.

\begin{figure}
    \centering
    \begin{subfigure}[b]{0.45\textwidth}
        \includegraphics[width=\textwidth,trim = 1cm 6.5cm 1cm 6.5cm,clip=true]{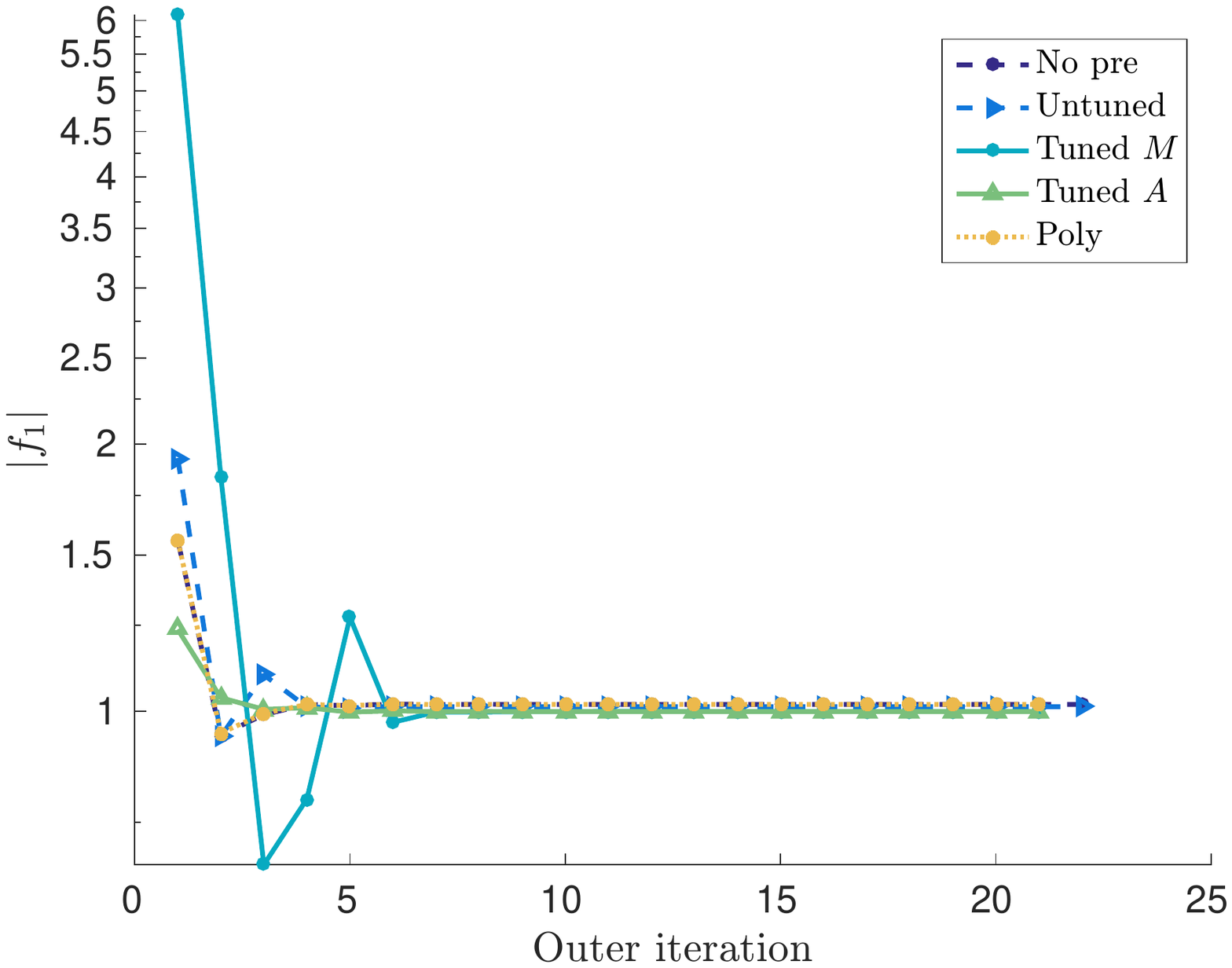}
    \end{subfigure}
    \begin{subfigure}[b]{0.45\textwidth}
        \includegraphics[width=\textwidth,trim = 1cm 6.5cm 1cm 6.5cm,clip=true]{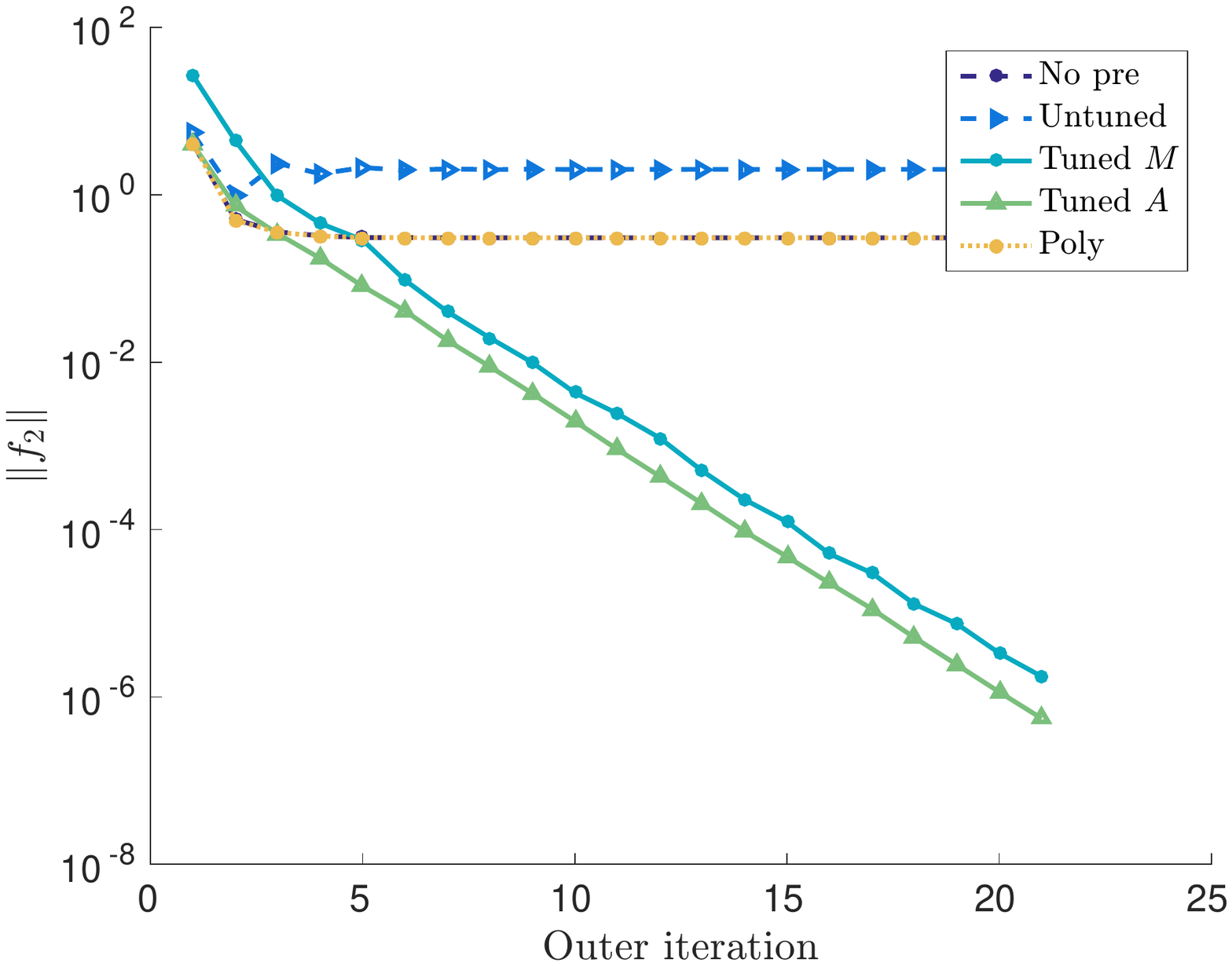}
    \end{subfigure}
    \begin{subfigure}[b]{0.45\textwidth}
        \includegraphics[width=\textwidth,trim = 1cm 6.5cm 1cm 6.5cm,clip=true]{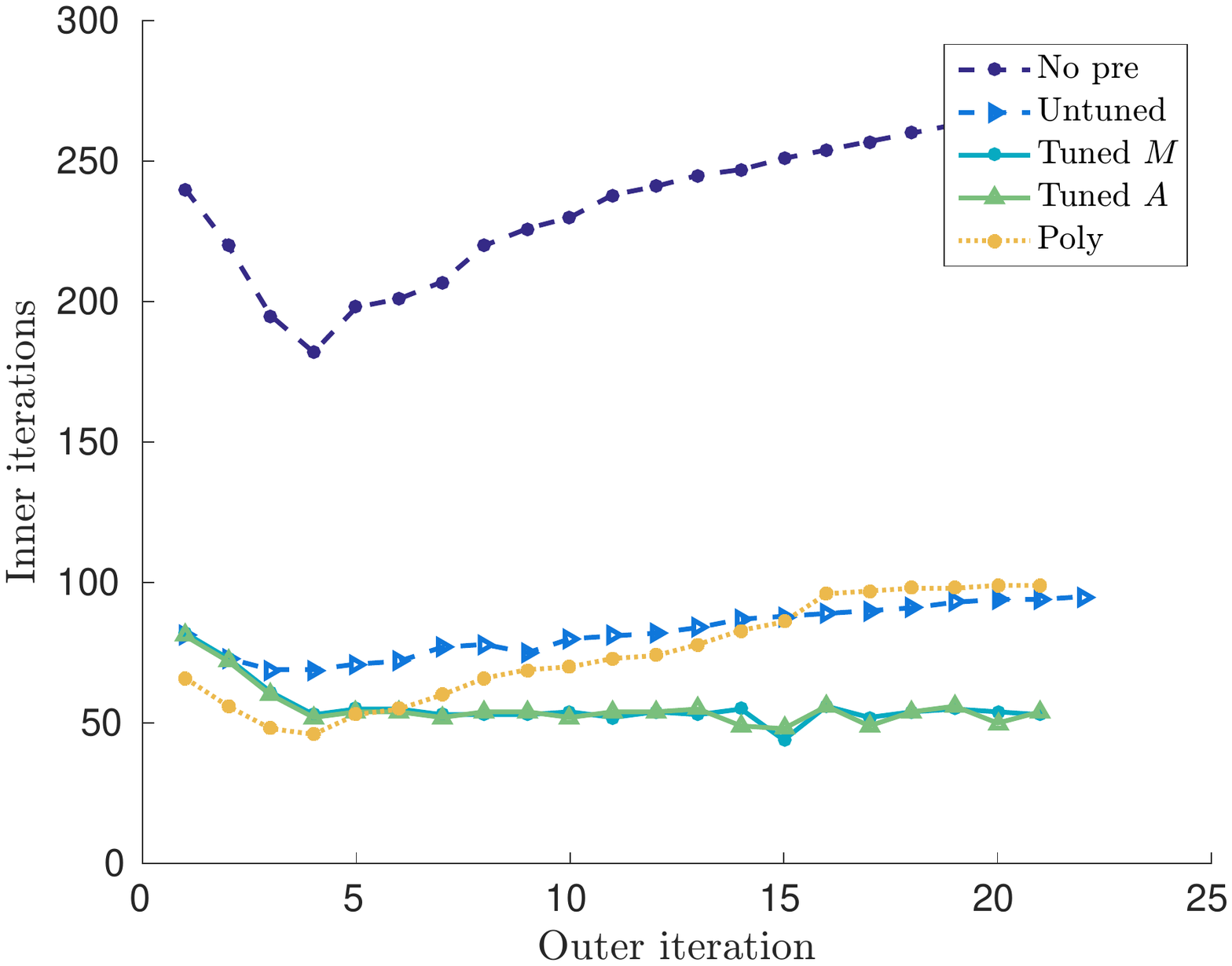}
    \end{subfigure}
        \begin{subfigure}[b]{0.45\textwidth}
        \includegraphics[width=\textwidth,trim = 1cm 6.5cm 1cm 6.5cm,clip=true]{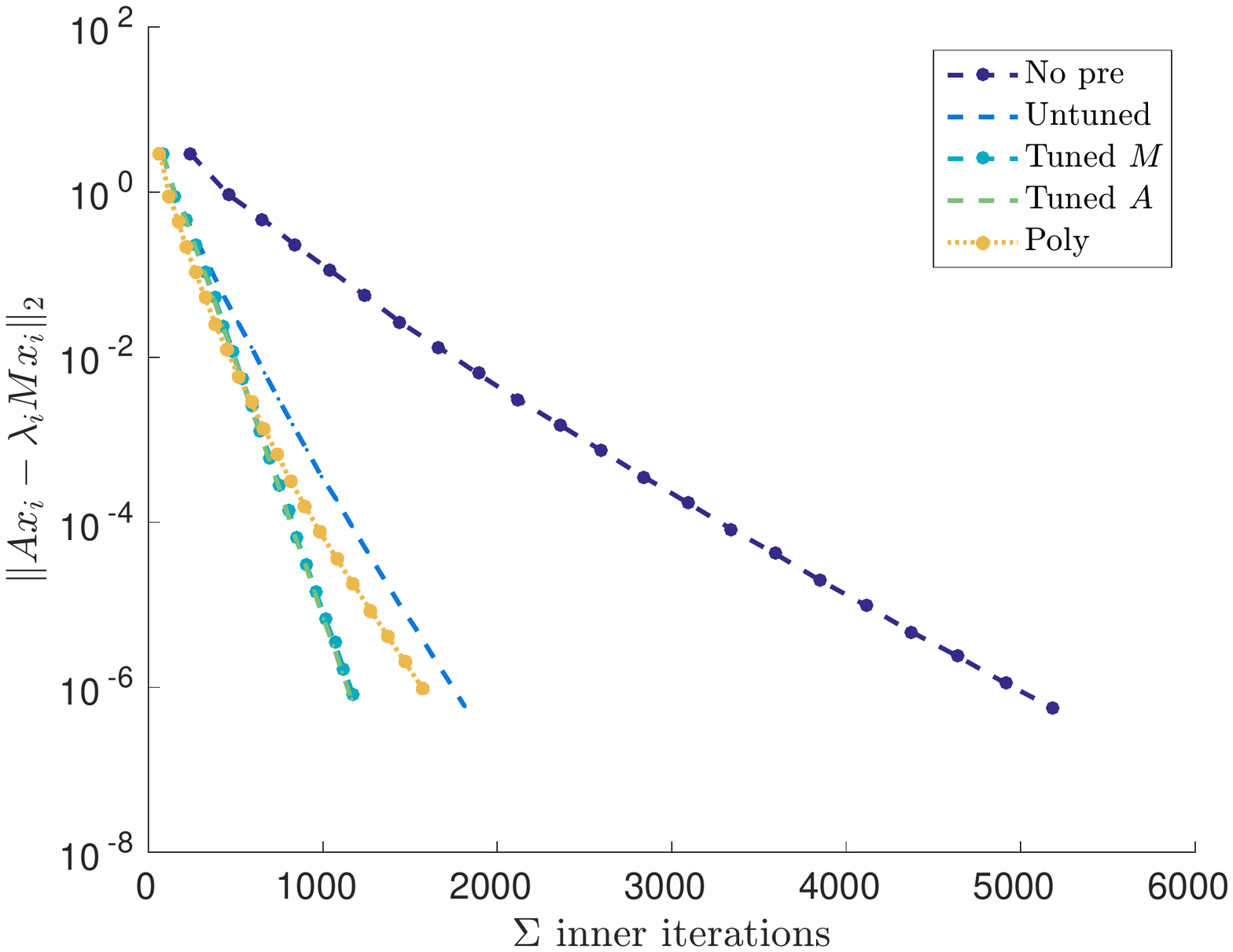}
    \end{subfigure}
    \caption{Results for the generalized problem $A=$\texttt{cd\_fd}, $M=\mathrm{tridiag}(2.5,5,2.5)$: weight components $|f_1|$, $\|f_2\|$, number of GMRES
iterations and eigenvalue residual. The ILU drop tolerance is $10^{-1}$ and deg$(p)=5$.}
     \label{f:Gcdfd}
\end{figure}

Even though this paper concentrates on GMRES bounds for inverse iteration, we finally show results of block-GMRES convergence for inverse subspace iteration. We
consider the matrix \cdde{} with the settings from above~and we seek an invariant subspace of dimension $u=6$. The drop tolerance of the incomplete LU
factorization used for the preconditioner is $10^{-2}$ and we apply a tuned preconditioner which satisfies $P_iY_i = AY_i$, e.g.\ $P_i = P+(A-P)Y_i(Y_i^T
Y_i)^{-1}Y_i^T$. For the implementation of block GMRES we used an algorithm discussed in \cite{Sood15,Sood16}.

Figure~\ref{f:block} shows the history of the norms of $W_1$ and $W_2$, the splittings of the matrix $W\in\mathbb{C}^{n\times u}$ discussed in
Section~\ref{sec:bgmres}, as the outer iteration proceeds. The bottom plots show the required number of block GMRES steps against the outer iteration (left
plot)
and the residual norm against the cumulative sum of inner block GMRES steps (right plot). As expected the reduction of inner iteration numbers by the tuned
preconditioner is apparent. The reduction of $\|W_2\|$ can be seen for the unpreconditioned as well as tuned case, the magnitude of $\|W_1\|$ is close to
2.449, but the overall behavior of the weights is similar to the single vector inverse iteration.   
\begin{figure}[t]
    \centering
    \begin{subfigure}[b]{0.45\textwidth}
        \includegraphics[width=\textwidth,trim = 1cm 6.5cm 1cm 6.5cm,clip=true]{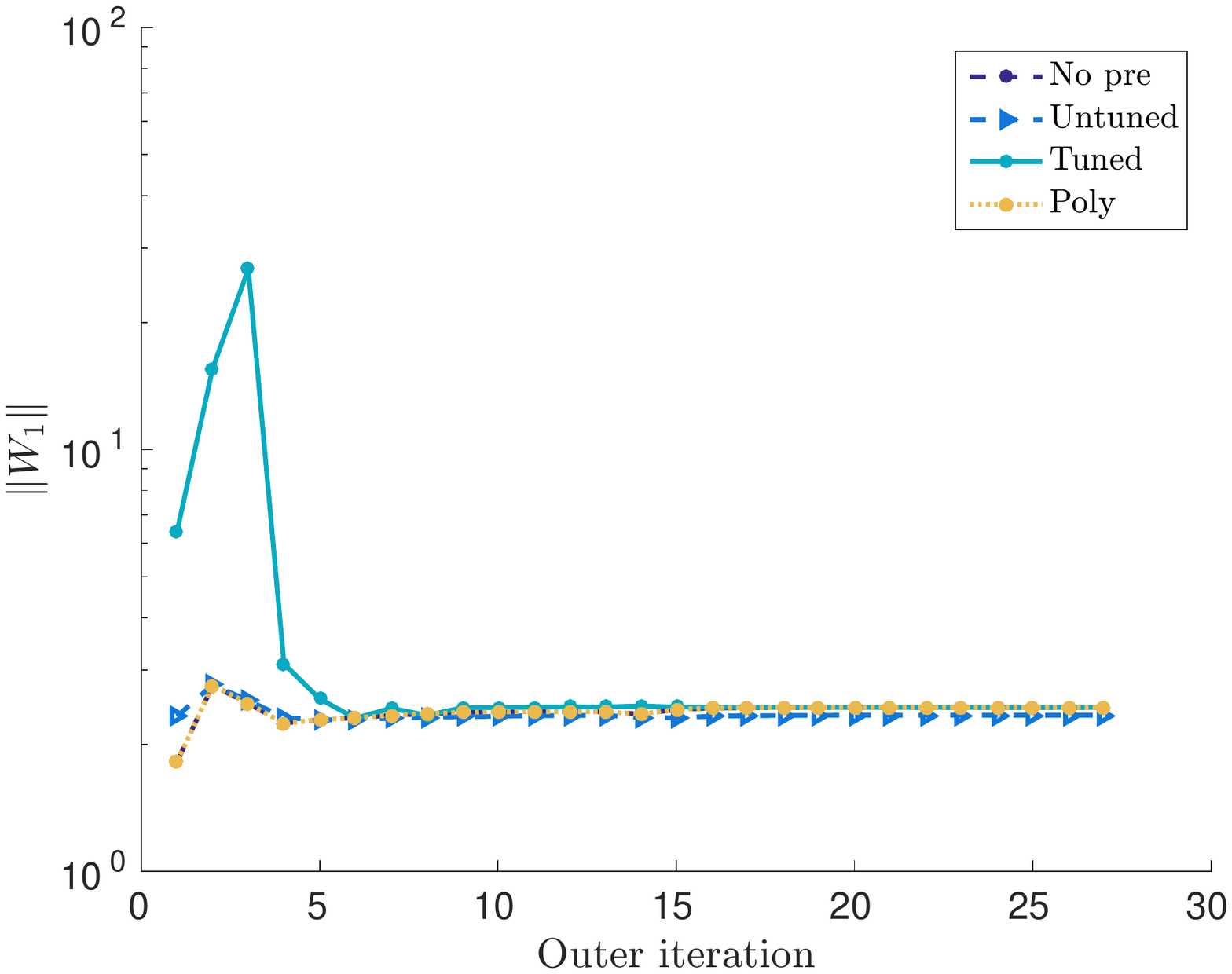}
    \end{subfigure}
    \begin{subfigure}[b]{0.45\textwidth}
       \includegraphics[width=\textwidth,trim = 1cm 6.5cm 1cm 6.5cm,clip=true]{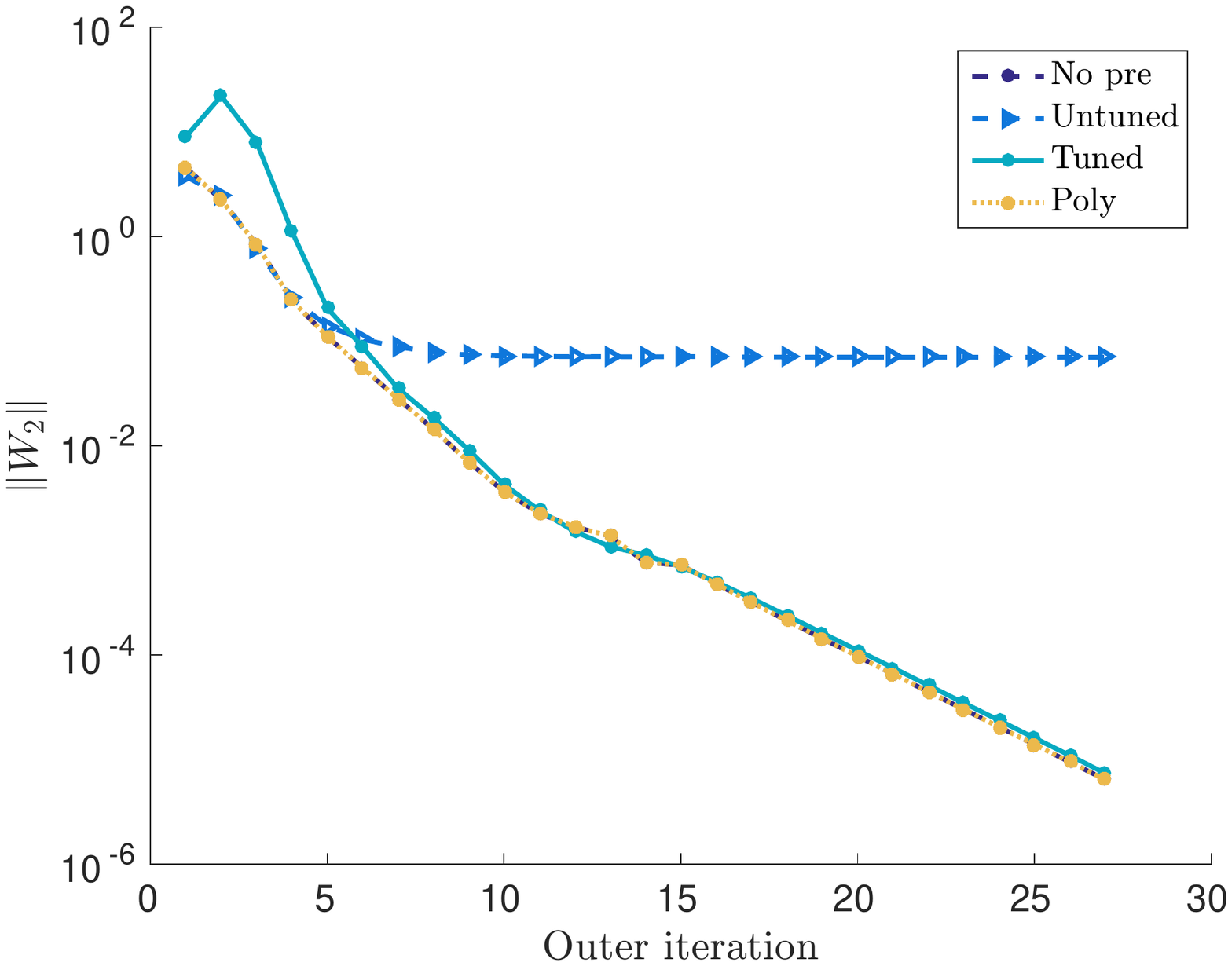}
    \end{subfigure}
    \begin{subfigure}[b]{0.45\textwidth}
        \includegraphics[width=\textwidth,trim = 1cm 6.5cm 1cm 6.5cm,clip=true]{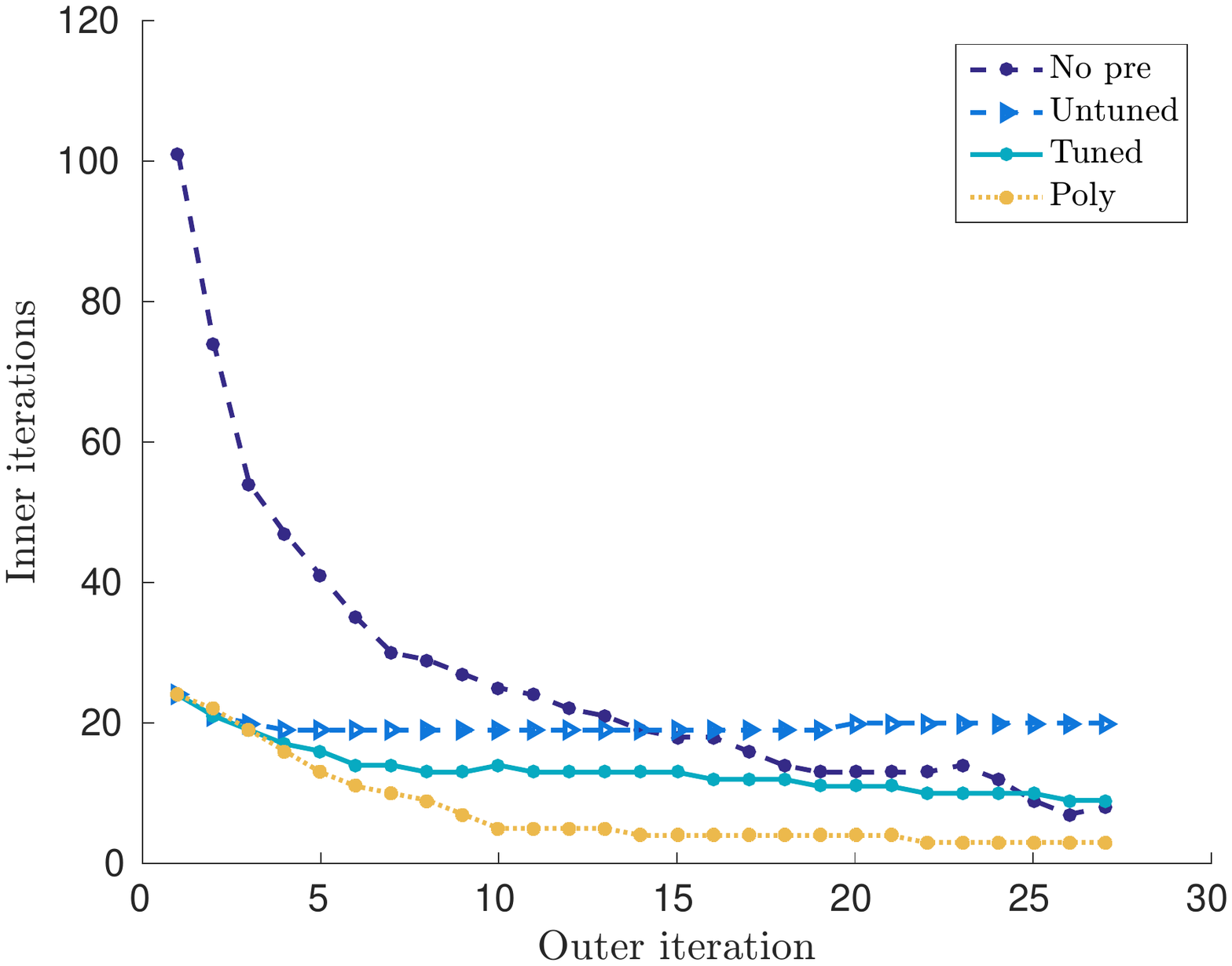}
    \end{subfigure}
        \begin{subfigure}[b]{0.45\textwidth}
        \includegraphics[width=\textwidth,trim = 1cm 6.5cm 1cm 6.5cm,clip=true]{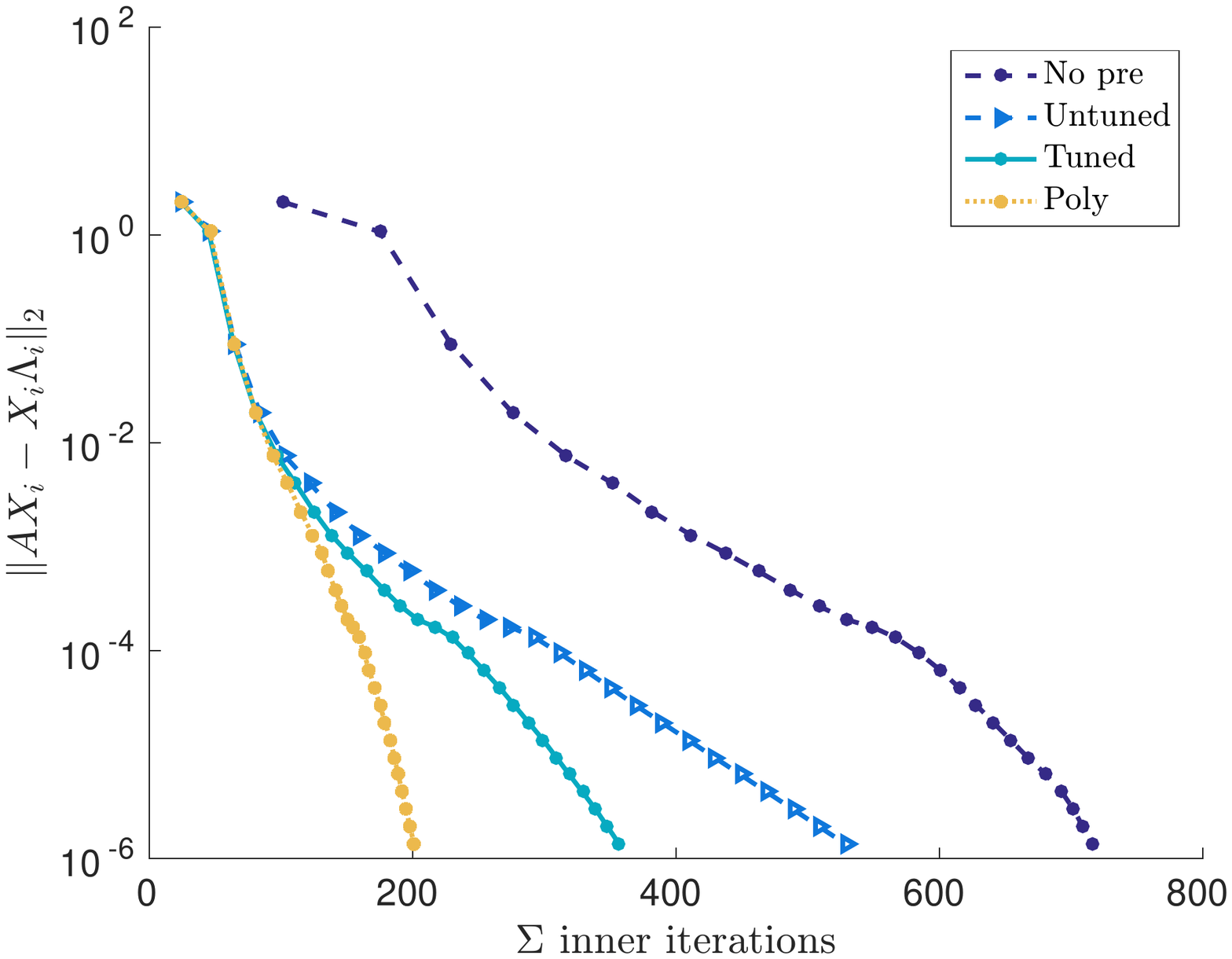}
    \end{subfigure}
    \caption{Inverse subspace iteration for \texttt{cdde1}, $u=6$: $\|W_1\|$, $\|W_2\|$, number of block-GMRES iterations and eigenvalue residual. The ILU drop 
tolerance is $10^{-2}$.}
     \label{f:block}
\end{figure}

\section{Conclusion}

In this paper we discussed the convergence behavior of GMRES (as a prominent iterative method) for solving linear systems that arise during the solution of
eigenvalue problems via inverse iteration. We gave detailed bounds on GMRES that take the special behavior of the right hand side into account and explained the
initial sharp decrease of the GMRES residual. The bounds gave rise to adapted preconditioners for GMRES when applied to eigenvalue problems, e.g. tuned and
polynomial preconditioners. The analysis was extended to inverse iteration for the generalized eigenvalue problem and subspace iteration, where block GMRES
bounds were given. The numerical results showed that the new GMRES (block GMRES) bounds are much sharper than  conventional bounds and that preconditioned
subspace iteration with either a tuned or polynomial preconditioner should be used, where the tuned preconditioner is generally easier
to construct. Possible future research perspectives should therefore, e.g., include generation strategies of high quality polynomial coefficients adapted to the
outer eigenvalue iteration.

\paragraph*{Acknowledgement.} 
 The authors thank Kirk Soodhalter (Radon Institute for
Computational and Applied Mathematics (RICAM)) for kindly providing his Matlab implementation of block-GMRES.
\input{gmres_eval_fin.bbl}

\end{document}

%% file: Tab/IICond.tex
\cdfd{} & 1024 & 4.6 & 1.2e+03& 7.6 & 1.5e+14& 5.7 & 3.4e+06& 6.9 & 7.1e+04\\
\cdde{} & 961 & 2.8 &  17& 7.4 & 5.5e+10& 5.2 & 1e+05& 4.3 & 1.8e+03\\
\olm{} & 2000 & 1.4 &  81&  11 & 5.1e+05&   7 & 3.1e+04& 2.7 & 1.3e+03\\

%% file: Tab/cdde11e-2Comb1.tex
\multirow{8}{*}{\rotatebox{90}{\mbox{Unprec.}}} & 1 & 0.099735 & 0.0033 & 1.1 & 1.4 & 5.1 & -- & -- & 54\\
 & 2 & -0.0028427 & 0.00095 & 0.9 & 0.16 & 2.6 & 3.7 & 1.1e+05 & 61\\
 & 3 & -0.0059742 & 0.00027 &   1 & 0.048 & 2.9 & 0.65 & 6.9e+04 & 51\\
 & 4 & -0.0049326 & 8e-05 & 0.99 & 0.012 & 2.8 & 0.16 & 6.1e+04 & 41\\
 & 5 & -0.0052396 & 2.2e-05 &   1 & 0.0035 & 2.8 & 0.045 & 5.7e+04 & 34\\
 & 6 & -0.0051539 & 6.1e-06 &   1 & 0.00094 & 2.8 & 0.012 & 5.6e+04 & 32\\
 & 7 & -0.0051775 & 1.7e-06 &   1 & 0.00026 & 2.8 & 0.0034 & 5.5e+04 & 28\\
 & 8 & -0.005171 & 4.6e-07 &   1 & 7e-05 & 2.8 & 0.00092 & 5.5e+04 & 25\\
\hline
\multirow{8}{*}{\rotatebox{90}{\mbox{Untuned}}} & 1 & 0.099735 & 0.0033 &  63 & 1.4e+02 & 8.1e+02 & -- & -- & 8\\
 & 2 & -0.0028504 & 0.00095 &   3 & 4.4 &  28 &   5 & 1.5e+05 & 11\\
 & 3 & -0.0059693 & 0.00027 & 2.7 & 4.1 &  26 & 5.6 & 5.9e+05 & 12\\
 & 4 & -0.0049345 & 7.9e-05 & 2.5 &   4 &  24 & 5.4 & 2e+06 & 13\\
 & 5 & -0.0052389 & 2.2e-05 & 2.5 &   4 &  25 & 5.5 & 6.9e+06 & 14\\
 & 6 & -0.0051541 & 6e-06 & 2.5 &   4 &  25 & 5.5 & 2.5e+07 & 15\\
 & 7 & -0.0051775 & 1.6e-06 & 2.5 &   4 &  25 & 5.5 & 9.1e+07 & 16\\
 & 8 & -0.0051711 & 4.5e-07 & 2.5 &   4 &  25 & 5.5 & 3.3e+08 & 17\\
\hline
\multirow{8}{*}{\rotatebox{90}{\mbox{Poly. prec.}}} & 1 & 0.099735 & 0.0033 & 1.1 & 1.4 & 5.1 & -- & -- & 9\\
 & 2 & -0.0028791 & 0.00094 & 0.91 & 0.16 & 2.6 & 3.6 & 1.1e+05 & 9\\
 & 3 & -0.0059578 & 0.00027 &   1 & 0.047 & 2.9 & 0.64 & 6.8e+04 & 6\\
 & 4 & -0.0049379 & 7.8e-05 & 0.99 & 0.012 & 2.8 & 0.16 & 6e+04 & 5\\
 & 5 & -0.005238 & 2.2e-05 &   1 & 0.0034 & 2.8 & 0.044 & 5.7e+04 & 4\\
 & 6 & -0.0051543 & 5.9e-06 &   1 & 0.00092 & 2.8 & 0.012 & 5.6e+04 & 4\\
 & 7 & -0.0051774 & 1.6e-06 &   1 & 0.00025 & 2.8 & 0.0033 & 5.5e+04 & 4\\
 & 8 & -0.0051711 & 4.5e-07 &   1 & 6.9e-05 & 2.8 & 0.0009 & 5.5e+04 & 3\\

%% file: Tab/cdde11e-2Comb2.tex
\multirow{8}{*}{$I$} & 1 & 0.099735 & 0.0033 & 4.9 & 1.9e+04 &  17 &  60 & 3.1e+02 & -- & -- & 9\\
 & 2 & -0.00285 & 0.00095 & 5.2 & 2.3e+04 & 0.98 & 0.73 & 6.3 & 5.4e+02 & 1.7e+07 & 11\\
 & 3 & -0.005968 & 0.00027 & 5.2 & 2.5e+04 &   1 & 0.012 & 5.2 & 7.4 & 7.8e+05 & 9\\
 & 4 & -0.0049348 & 7.9e-05 & 5.2 & 2.4e+04 &   1 & 0.0028 & 5.2 & 1.4 & 5.2e+05 & 8\\
 & 5 & -0.0052389 & 2.2e-05 & 5.2 & 2.4e+04 &   1 & 0.00076 & 5.2 & 0.4 & 5e+05 & 8\\
 & 6 & -0.0051541 & 6e-06 & 5.2 & 2.4e+04 &   1 & 0.00021 & 5.2 & 0.11 & 5e+05 & 8\\
 & 7 & -0.0051775 & 1.6e-06 & 5.2 & 2.4e+04 &   1 & 5.8e-05 & 5.2 & 0.03 & 5e+05 & 8\\
 & 8 & -0.0051711 & 4.5e-07 & 5.2 & 2.4e+04 &   1 & 1.6e-05 & 5.2 & 0.0083 & 5e+05 & 8\\
\hline
\multirow{8}{*}{$A$} & 1 & 0.099735 & 0.0033 &   5 & 2.5e+04 &  15 &  46 & 2.4e+02 & -- & -- & 8\\
 & 2 & -0.0028498 & 0.00095 & 5.2 & 2e+04 & 2.5 & 5.9 &  33 & 5.6 & 1.7e+05 & 10\\
 & 3 & -0.0059681 & 0.00027 & 5.2 & 2.5e+04 &   1 & 0.18 & 5.3 & 0.7 & 7.4e+04 & 8\\
 & 4 & -0.0049348 & 7.9e-05 & 5.2 & 2.4e+04 &   1 & 0.048 & 5.2 & 0.2 & 7.5e+04 & 7\\
 & 5 & -0.0052389 & 2.2e-05 & 5.2 & 2.4e+04 &   1 & 0.013 & 5.2 & 0.054 & 6.8e+04 & 7\\
 & 6 & -0.0051541 & 6e-06 & 5.2 & 2.4e+04 &   1 & 0.0035 & 5.2 & 0.015 & 6.8e+04 & 7\\
 & 7 & -0.0051775 & 1.6e-06 & 5.2 & 2.4e+04 &   1 & 0.00095 & 5.2 & 0.004 & 6.7e+04 & 7\\
 & 8 & -0.0051711 & 4.5e-07 & 5.2 & 2.4e+04 &   1 & 0.00026 & 5.2 & 0.0011 & 6.7e+04 & 7\\

%% file: Tab/results_all.tex
\multirow{13}{*}{\rotatebox{90}{\mbox{\cdfd{}}}}& \multicolumn{2}{c}{no prec.} &      1 & 8.2e-08 & 15 & 305\\
&std. $P$& 1e-1& 1.6735e+08 & 1.8e+08 & 13 & 225\\
&tuned $I$& 1e-1& 58.199 & 1.1e+02 & 15 & 100\\
&tuned $A$& 1e-1& 1.7041 & 2.3 & 15 & 87\\
&poly& $d=5$&      1 & 1.6e-07 & 14 & 61\\
\cline{2-7}&std. $P$& 1e-2& 44.229 &  87 & 15 & 148\\
&tuned $I$& 1e-2&      1 & 3.6e-05 & 14 & 53\\
&tuned $A$& 1e-2&      1 & 1e-06 & 14 & 39\\
&poly& $d=10$&      1 & 1.4e-07 & 14 & 40\\
\cline{2-7}&std. $P$& 1e-3& 1.2081 & 4.4 & 15 & 89\\
&tuned $I$& 1e-3&      1 & 2.1e-06 & 14 & 45\\
&tuned $A$& 1e-3&      1 & 2.4e-07 & 14 & 31\\
&poly& $d=15$&      1 & 4.7e-08 & 20 & 79\\
\hline\multirow{13}{*}{\rotatebox{90}{\mbox{\cdde{}}}}& \multicolumn{2}{c}{no prec.} & 0.99995 & 7e-05 & 8 & 326\\
&std. $P$& 1e-1& 6.0179e+05 & 1.2e+06 & 8 & 190\\
&tuned $I$& 1e-1& 0.99997 & 1.2 & 8 & 123\\
&tuned $A$& 1e-1& 1280.8 & 1.4e+03 & 8 & 115\\
&poly& $d=5$& 0.99995 & 6.9e-05 & 8 & 67\\
\cline{2-7}&std. $P$& 1e-2& 2.5239 &   4 & 8 & 106\\
&tuned $I$& 1e-2& 0.99999 & 1.6e-05 & 8 & 69\\
&tuned $A$& 1e-2&      1 & 0.00026 & 8 & 61\\
&poly& $d=10$& 0.99995 & 6.9e-05 & 8 & 44\\
\cline{2-7}&std. $P$& 1e-3& 3.4135 &  10 & 8 & 53\\
&tuned $I$& 1e-3&      1 & 3.9e-05 & 8 & 40\\
&tuned $A$& 1e-3& 0.99994 & 0.00026 & 8 & 32\\
&poly& $d=15$& 0.99998 & 3.7e-05 & 10 & 51\\
\hline\multirow{13}{*}{\rotatebox{90}{\mbox{\olm{}}}}& \multicolumn{2}{c}{no prec.} &      1 & 1e-06 & 8 & 3172\\
&std. $P$& 1e-1& 3.4649 & 2.8 & 14 & 505\\
&tuned $I$& 1e-1&      1 & 9.2e-06 & 14 & 393\\
&tuned $A$& 1e-1&      1 & 1.2e-05 & 14 & 379\\
&poly& $d=5$&      1 & 9.7e-07 & 8 & 1934\\
\cline{2-7}&std. $P$& 1e-2& 0.97526 & 0.43 & 11 & 365\\
&tuned $I$& 1e-2&      1 & 4.6e-06 & 11 & 303\\
&tuned $A$& 1e-2&      1 & 2e-06 & 12 & 318\\
&poly& $d=10$&      1 & 1.2e-06 & 8 & 416\\
\cline{2-7}&std. $P$& 1e-3& 0.97524 & 0.12 & 10 & 313\\
&tuned $I$& 1e-3&      1 & 4.3e-06 & 10 & 260\\
&tuned $A$& 1e-3&      1 & 2.9e-06 & 10 & 255\\
&poly& $d=15$&      1 & 1.7e-05 & 20 & 1660\\
\hline

%% file: gmres_eval_fin.bbl
\newcommand{\noopsort}[1]{} \newcommand{\printfirst}[2]{\#1}
  \newcommand{\singleletter}[1]{\#1} \newcommand{\switchargs}[2]{\#2\#1}

%% file: gmres_eval_fin.bbl
\begin{thebibliography}{10}

\bibitem{AhmSG12}
M.~I. Ahmad, D.~B. Szyld, and M.~B. van Gijzen.
\newblock {Preconditioned multishift BiCG for {{\cal H}}$_2$-optimal model
  reduction}.
\newblock {\em {{SIAM} J. Matrix Anal. Appl.}}, 38(2):401--424, 2017.

\bibitem{APS98}
M.~Arioli, V.~Pt{\'a}k, and Z.~Strako\v{s}.
\newblock {Krylov sequences of maximal length and convergence of {GMRES}}.
\newblock {\em BIT}, 38:636--643, 1998.

\bibitem{AshMO92}
S.~F. Ashby, T.~A. Manteuffel, and J.~S. Otto.
\newblock {A Comparison of Adaptive Chebyshev and Least Squares Polynomial
  Preconditioning for Hermitian Positive Definite Linear Systems}.
\newblock {\em SIAM J. Sci. Statist. Comput.}, 13(1):1--29, 1992.

\bibitem{templatesLS}
R.~Barrett, M.~Berry, T.~F. Chan, J.~Demmel, J.~Donato, J.~Dongarra,
  V.~Eijkhout, R.~Pozo, C.~Romine, and H.~A. van~der Vorst.
\newblock {\em Templates for the Solution of Linear Systems: Building Blocks
  for Iterative Methods, 2nd Edition}.
\newblock SIAM, Philadelphia, PA, 1994.

\bibitem{BG15}
M.~Baumann and M.~van Gijzen.
\newblock {Nested {K}rylov methods for shifted linear systems}.
\newblock {\em SIAM J. Sci. Comput.}, 37:90--112, 2015.

\bibitem{BeGrSp04}
J.~Berns-M{\"u}ller, I.~G. Graham, and A.~Spence.
\newblock Inexact inverse iteration for symmetric matrices.
\newblock {\em Linear Algebra Appl.}, 416(2):389--413, 2006.

\bibitem{DuSZ15}
L.~Du, T.~Sogabe, and S.-L. Zhang.
\newblock {IDR($s$) for solving shifted nonsymmetric linear systems}.
\newblock {\em J. Comput. Appl. Math.}, 274(0):35--43, 2015.

\bibitem{DTM14}
J.~{Duintjer Tebbens} and G.~Meurant.
\newblock {Prescribing the behavior of early terminating {GMRES} and {A}rnoldi
  iterations}.
\newblock {\em Numer. Algorithms}, 65:69--90, 2014.

\bibitem{frei07}
M.~A. Freitag.
\newblock {\em {Inner-outer Iterative Methods for Eigenvalue Problems --
  Convergence and Preconditioning}}.
\newblock PhD thesis, University of Bath, 2007.

\bibitem{FrSp06a}
M.~A. Freitag and A.~Spence.
\newblock Convergence theory for inexact inverse iteration applied to the
  generalised nonsymmetric eigenproblem.
\newblock {\em Electron. Trans. Numer. Anal.}, 28:40--64, 2007.

\bibitem{FrSp05b}
M.~A. Freitag and A.~Spence.
\newblock A tuned preconditioner for inexact inverse iteration applied to
  {H}ermitian eigenvalue problems.
\newblock {\em IMA J. Numer. Anal.}, 28(3):522--551, 2008.

\bibitem{FrSp07b}
M.~A. Freitag, A.~Spence, and E.~Vainikko.
\newblock Rayleigh quotient iteration and simplified {J}acobi-{D}avidson with
  preconditioned iterative solves for generalised eigenvalue problems.
\newblock Techn. report, Dept. of Math. Sciences, University of Bath, 2008.

\bibitem{Freund1990}
R.~Freund.
\newblock {On conjugate gradient type methods and polynomial preconditioners
  for a class of complex non-hermitian matrices}.
\newblock {\em Numer. Math.}, 57(1):285--312, 1990.

\bibitem{GolubYe00}
G.~H. Golub and Q.~Ye.
\newblock Inexact inverse iteration for generalized eigenvalue problems.
\newblock {\em BIT}, 40(4):671--684, 2000.

\bibitem{Ipsen97}
I.~C. Ipsen.
\newblock {Computing an eigenvector with inverse iteration}.
\newblock {\em SIAM Rev.}, 39(2):254--291, 1997.

\bibitem{LiuMW15}
Q.~Liu, R.~B. Morgan, and W.~Wilcox.
\newblock {Polynomial Preconditioned GMRES and GMRES-DR}.
\newblock {\em SIAM J. Sci. Comput.}, 37(5):S407--S428, 2015.

\bibitem{Mar16}
A.~Martinez.
\newblock {Tuned preconditioners for the eigensolution of large SPD matrices
  arising in engineering problems}.
\newblock {\em Numer. Linear Algebra Appl.}, 23(3):427--443, 2016.

\bibitem{MDT15}
G.~Meurant and J.~{Duintjer Tebbens}.
\newblock {The role eigenvalues play in forming {GMRES} residual norms with
  non-normal matrices}.
\newblock {\em Numer. Algorithms}, 68:143--165, 2015.

\bibitem{PaiS75}
C.~C. Paige and M.~A. Saunders.
\newblock {Solution of Sparse Indefinite Systems of Linear Equations}.
\newblock {\em {{SIAM} J. Numer. Anal.}}, 12(4):617--629, 1975.

\bibitem{Sood16}
M.~L. Parks, K.~M. Soodhalter, and D.~B. Szyld.
\newblock {A block Recycled GMRES method with investigations into aspects of
  solver performance}.
\newblock {\em ArXiv e-prints (and submitted for publication)}, 2016.

\bibitem{PetW93}
G.~Peters and J.~H. Wilkinson.
\newblock Inverse iteration, ill-conditioned equations and {N}ewton's method.
\newblock {\em SIAM Rev.}, 21(3):339--360, 1979.

\bibitem{RSS06}
M.~Robb{\'e}, M.~Sadkane, and A.~Spence.
\newblock Inexact inverse subspace iteration with preconditioning applied to
  non-hermitian eigenvalue problems.
\newblock {\em SIAM J. Matrix Anal. Appl.}, 31(1):92--113, Feb. 2009.

\bibitem{Saa93}
Y.~Saad.
\newblock {A Flexible Inner-Outer Preconditioned GMRES Algorithm}.
\newblock {\em SIAM J. Sci. Comput.}, 14(2):461--469, 1993.

\bibitem{SaadSchultz86}
Y.~Saad and M.~Schultz.
\newblock {GMRES} a generalised minimum residual algorithm for solving
  nonsymmetric linear systems.
\newblock {\em SIAM J. Sci. Statist. Comput.}, 7:856--869, 1986.

\bibitem{SimonElden02}
V.~Simoncini and L.~Eld{\'e}n.
\newblock Inexact {R}ayleigh quotient-type methods for eigenvalue computations.
\newblock {\em BIT}, 42(1):159--182, 2002.

\bibitem{Sood15}
K.~Soodhalter.
\newblock {A block MINRES algorithm based on the banded Lanczos method}.
\newblock {\em Numer. Algorithms}, 69:473--494, 2015.

\bibitem{SzyX11}
D.~Szyld and F.~Xue.
\newblock {Efficient {P}reconditioned {I}nner {S}olves {F}or {I}nexact
  {R}ayleigh {Q}uotient {I}teration And Their {C}onnections To The
  {S}ingle-{V}ector {J}acobi--{D}avidson {M}ethod}.
\newblock {\em SIAM J. Matrix Anal. A.}, 32(3):993--1018, 2011.

\bibitem{TPPW14}
D.~Titley-Peloquin, J.~Pestana, and A.~J.~Wathen.
\newblock {{GMRES} convergence bounds that depend on the right-hand-side
  vector}.
\newblock {\em IMA J. Numer. Anal.}, 34:462--479, 2014.

\bibitem{Gij95}
M.~B. van Gijzen.
\newblock {A polynomial preconditioner for the {GMRES} algorithm}.
\newblock {\em J. Comput. Appl. Math.}, 59(1):91--107, 1995.

\bibitem{GijSZ15}
M.~B. van Gijzen, G.~L.~G. Sleijpen, and J.-P.~M. Zemke.
\newblock {Flexible and multi-shift induced dimension reduction algorithms for
  solving large sparse linear systems}.
\newblock {\em {Numer. Lin. Alg. Appl.}}, 22(1):1--25, 2015.

\bibitem{ElmXue11}
F.~Xue and H.~C. Elman.
\newblock {Fast inexact subspace iteration for generalized eigenvalue problems
  with spectral transformation}.
\newblock {\em Linear Algebra Appl.}, 435(3):601--622, 2011.

\end{thebibliography}
